\documentclass[12pt]{amsart}
\usepackage[top=1in, bottom=1in, margin=1in]{geometry}
\usepackage{amsfonts,amsmath,amssymb,amscd,amsthm}
\usepackage[all]{xy}
\usepackage{enumerate}
\usepackage{verbatim}

\newtheorem{theorem}{Theorem}[section]
\newtheorem{prop}[theorem]{Proposition}
\newtheorem{lem}[theorem]{Lemma}
\newtheorem{cor}[theorem]{Corollary}
\newtheorem{rem}[theorem]{Remark}
\newtheorem{mydef}[theorem]{Definition}

\begin{document}

\title{Borel and Continuous Systems of Measures}
\author{A. Censor}
\address{Aviv Censor, Department of Mathematics,
University of California at Riverside, Riverside, CA, 92521, U.S.A.}
\email{avivc@math.ucr.edu}
\author{D. Grandini}
\address{Daniele Grandini, Department of Mathematics,
University of California at Riverside, Riverside, CA, 92521, U.S.A.}
\email{daniele@math.ucr.edu}
\date{\today}
\begin{abstract}
We study Borel systems and continuous systems of measures, with a focus on mapping properties: compositions, liftings, fibred products and disintegration. Parts of the theory we develop can be derived from known work in the literature,
and in that sense this paper is of expository nature. However, we put the above notions in the spotlight and provide a self-contained, purely measure-theoretic, detailed and thorough investigation of their properties, and in that aspect our paper enhances and complements the existing literature. Our work constitutes part of the necessary theoretical framework for categorical constructions involving measured and topological groupoids with Haar systems, a line of research we pursue in separate papers.
\end{abstract}
\maketitle

\section{Introduction}

We first give an overview of the contents of this paper. This is followed by a discussion of the nature of our work and its relation to the existing literature.

\subsection{Overview}
Our treatment of Borel systems of measures (BSMs) and continuous systems of measures (CSMs) in this paper is very general. Loosely speaking, a system of measures on a map $\pi:X \rightarrow Y$ is a family of measures $\lambda^{\bullet} = \{ \lambda^y \}_{y \in Y}$ on $X$, such that each $\lambda^y$ is concentrated on $\pi^{-1}(y)$. This can be made precise when the nature of $X$, $Y$ and $\pi$ is specified (e.g. topological spaces with a continuous map, Borel spaces with a Borel map), leading to appropriate assumptions on the measures $\{ \lambda^y \}$. We will denote a map $\pi:X \rightarrow Y$ admitting a system of measures $\lambda^{\bullet}$ by the diagram $\xymatrix{X\ar [rr]^{\pi}_{\lambda^{\bullet}}&&Y}$.

In the spotlight of our work are mapping properties of systems of measures. We establish terminology, notation and basic properties of systems of measures in Section \ref{sec:SOM}. Then, in Section \ref{sec:composition}, we study \emph{composition} of systems, corresponding to the following diagram $$\xymatrix{X\ar [rr]^{p}_{\alpha^{\bullet}}&&Y\ar [rr]^{q}_{\beta^{\bullet}}&&Z}$$ The composition $(\beta\circ\alpha)^{\bullet}$ is defined for any Borel set $E \subseteq X$ by $(\beta\circ\alpha)^{z}(E) =\int_Y \alpha^{y}(E) \ d\beta^z(y)$ (Definition \ref{def:composition BSM}).

In Section \ref{sec:lifting} we treat the notion of \emph{lifting}, namely producing a system of measures $(q^*\alpha)^{\bullet}$ on $\pi_Y$ in the following pull-back diagram:
$$\xymatrix{X * Y \ar [dd]_{\pi_X}\ar [rr]^{\pi_Y}&&Y\ar [dd]_{q}\\\\
X\ar [rr]^{p}_{\alpha^{\bullet}}&&Z}$$ The lifting is given by $(q^*\alpha)^y=\alpha^{q(y)}\times \delta_y$ (Definition \ref{def:lifting BSM}).

Section \ref{sec:fibred products} deals with the \emph{fibred product}, which is a system of measures $(\gamma_X * \gamma_Y)^{\bullet}$ on the map $f*g$ in the following diagram:
\[
\xy 0;<.2cm,0cm>:
(20,20)*{X_2 * Y_2}="1";
(40,20)*{Y_2}="2";
(20,0)*{X_2}="3";
(40,0)*{Z}="4";
(10,10)*{X_1 * Y_1}="5";
(30,10)*{Y_1}="6";
(10,-10)*{X_1}="7";
(30,-10)*{Z}="8";
{"1"+CR+(.5,0);"2"+CL+(-.5,0)**@{-}?>*{\dir{>}}?>(.5)+(0,1)*{\scriptstyle \pi_{Y_2}}};
{"5"+CR+(.5,0);"6"+CL+(-.5,0)**@{-}?>*{\dir{>}}?>(.75)+(0,1)*{\scriptstyle \pi_{Y_1}}};
{"3"+CR+(.5,0);"4"+CL+(-10,0)**@{-}};
{"3"+CR+(9,0);"4"+CL+(-.5,0)**@{-}?>*{\dir{>}}?>(.5)+(0,1)*{\scriptstyle p_2}};
{"7"+CR+(.5,0);"8"+CL+(-.5,0)**@{-}?>*{\dir{>}}?>(.5)+(0,1)*{\scriptstyle p_1}};
{"1"+CD+(0,-.5);"3"+CU+(0,9.6)**@{-}} ;
{"1"+CD+(0,-9.8);"3"+CU+(0,.5)**@{-}?>*{\dir{>}}?>(.5)+(-1.5,0)*{\scriptstyle \pi_{X_2}}} ;
{"2"+CD+(0,-.5);"4"+CU+(0,.5)**@{-}?>*{\dir{>}}?>(.5)+(-1,0)*{\scriptstyle q_2}} ;
{"5"+CD+(0,-.5);"7"+CU+(0,.5)**@{-}?>*{\dir{>}}?>(.5)+(-1.5,0)*{\scriptstyle \pi_{X_1}}} ;
{"6"+CD+(0,-.5);"8"+CU+(0,.5)**@{-}?>*{\dir{>}}?>(.25)+(-1,0)*{\scriptstyle q_1}} ;
{"5"+C+(1.4,1.4);"1"+C+(-1.4,-1.4)**@{-}?>*{\dir{>}}?>*{\dir{>}}?>(.5)+(-2,0)*{\scriptstyle f*g}} ;
{"6"+C+(1.4,1.4);"2"+C+(-1.4,-1.4)**@{-}?>*{\dir{>}}?>(.5)+(-1,0)*{\scriptstyle g}?>(.5)+(1.5,0)*{\scriptstyle \gamma_Y^{\bullet}}} ;
{"7"+C+(1.4,1.4);"3"+C+(-1.4,-1.4)**@{-}?>*{\dir{>}}?>(.5)+(-1,0)*{\scriptstyle f}?>(.5)+(1.5,0)*{\scriptstyle \gamma_X^{\bullet}}} ;
{"8"+C+(1.4,1.4);"4"+C+(-1.4,-1.4)**@{-}?>*{\dir{>}}?>(.5)+(-1.5,0)*{\scriptstyle id}};
\endxy
\]
The fibred product is defined by $\left(\gamma_X * \gamma_Y\right)^{(x_2,y_2)}=\gamma_X^{x_2} \times \gamma_Y^{y_2}$ (Definition \ref{def:fibred product BSM}).

Section \ref{sec:disintegration} explores the concept of \emph{disintegration}, a most valuable tool in applications: If $(X,\mu)$ and $(Y,\nu)$ are measure spaces, and $f:X \rightarrow Y$ is a Borel map, then a system of measures $\gamma^{\bullet}$ on $f$ is a disintegration of $\mu$ with respect to $\nu$ if $ \mu (E) = \int_Y \gamma^y (E) d\nu (y)$ for every Borel set $E \subseteq X$.

We conclude, in section \ref{sec:groupoids}, with a brief discussion of systems of measures for groupoids, in particular Haar systems.

\subsection{Broad perspective}
While our interest in systems of measures originated from our work with groupoids, in this paper we develop the theory from elementary principles and our approach is purely measure theoretic. This is in contrast to many references where the subject has been studied from very specialized perspectives. Systems of measures (also called $\pi$-systems or kernels) appear in various mathematical contexts, and have been investigated from different viewpoints in the literature. For example, a general introduction to the topic can be found in Bourbaki \cite{bourbaki}, which takes a very functional analytic approach.

The primary goal we set for this paper was to collect and clarify the categorically-flavored constructions that we needed, details of which we managed to trace only in part in the functional analysis, probability and groupoid literature. We do not claim to present an exhaustive account of the literature on systems of measures.

The world of groupoids, which motivated our study, is a discipline in which systems of measures play a fundamental role. Most notably, a Haar system for a groupoid $G$ is essentially a left-invariant system of measures on the range map $r:G \rightarrow G^{(0)}$, which generalizes the notion of a Haar measure on a group. In particular, Haar systems are a crucial ingredient for integration on groupoids, for groupoid representations, and for constructing groupoid $C^*$-algebras. Beyond Haar systems, maps between groupoids naturally give rise to systems of measures as well.

In the groupoid literature, systems of measures have been studied extensively, for example by Connes in \cite{connes-noncommutative-integration} (using the term ``kernel", noyau in French), by Muhly in \cite{muhly-book-unpublished} and by Renault and Anantharaman-Delaroche in \cite{renault-anantharaman-delaroche} (using the term ``$\pi$-systems"). The scope of our current study of systems of measures was therefore restricted to mapping properties which were essential for specific applications that came up in our work. Some of the results presented here appear scattered across the literature, which is why we opted to give a self contained treatment, including all definitions, and full proofs whenever lacking precise references. We point out that some of the formulas and diagrams which we make explicit, can be found in \cite{renault-anantharaman-delaroche}. In fact, significant parts of the theory are implicit in, and can be non-trivially derived from the aforementioned groupoid references, as well as other works of Renault (e.g. \cite{renaultJOT}), Ramsay (e.g. \cite{ramsay71}) and others. We single out a couple of such sources which we refer the specialized reader to: The first is Appendix A.1 of \cite{renault-anantharaman-delaroche} on transverse measure theory, which builds on Connes' work, starting with \cite{connes-noncommutative-integration}. The second is a fibred product construction beginning on page 265 of \cite{ramsay71}. A detailed discussion of how to extract some of our results from these is beyond the scope of this paper.

This paper provides tools and techniques that allow us to form certain categorical constructions with topological groupoids, which we shall present in separate papers. Primarily, we were interested in forming the so-called ``weak pull-back" of a diagram of topological groupoids, each endowed with a Haar system and a quasi-invariant measure on its unit space \cite{WPB}. The weak pull-back is a key ingredient for degroupoidification \`a la Baez and Dolan \cite{baez-hoffnung-walker}, which together with Christopher Walker we are currently generalizing from the discrete setting to the realm of topology and measure theory.

\section{Systems of measures}\label{sec:SOM}

\textbf{\emph{Throughout, we will assume all topological spaces to be second countable and $\mathbf{T_1}$. We require spaces to also be locally compact and Hausdorff whenever dealing with continuous systems of measures, as well as throughout Section \ref{sec:disintegration}.}} Measures will always be positive and Borel. Unless stated otherwise, continuous functions will be complex-valued, whereas Borel functions are allowed to take infinite values. \\

We first recall the definition of the \textbf{support} of a Borel measure $\mu$ on a space $X$: $$supp(\mu)=\{x \in X: \mu(A)> 0\ \ \text{for every open neighborhood } A \text{ of } x \}.$$ We say that the measure $\mu$ is \textbf{concentrated} on a subset $S \subseteq X$ if $\mu(X \setminus S)=0$.

\begin{lem}\label{lem:support}
The support is a closed subset of $X$. Moreover, if $S$ is a closed subset of $X$, then $supp(\mu)\subseteq S$ if and only if the measure $\mu$ is concentrated on $S$.
\end{lem}

\begin{proof}
Take $x\notin supp(\mu)$. Then $x$ has an open neighborhood $A$ such that $\mu(A)=0$. Furthermore, $A \cap supp(\mu) = \emptyset$. This shows that the complement of $supp(\mu)$ is open.

For the second part, note first that $supp(\mu)\subseteq S$ if and only if
$$x\notin S \quad \Rightarrow \quad \exists A\subseteq X\text{ open }:\ x\in A,\ \mu(A)=0.$$
Assume that $\mu(X \setminus S)=0$. Since the complement $X \setminus S$ is open, $A=X \setminus S$ satisfies the above statement for any $x\notin S$ and it follows that $supp(\mu)\subseteq S$. \\
Viceversa, assume $supp(\mu)\subseteq S$. Fix a countable basis ${\mathcal B}$ for the topology of $X$. Then the following statement is true:
$$x\notin S\quad\Rightarrow\quad \exists A_x\in {\mathcal B}:\ x\in A_x,\ \mu(A_x)=0.$$ It follows that $X \setminus S \subseteq \bigcup_{x\notin S}A_x$. But this union consists of countably many distinct elements of the basis ${\mathcal B}$, so we can invoke countable subadditivity to obtain
$\mu(X \setminus S)\leq \sum_{x\notin S}\mu(A_x)=0.$
\end{proof}

\begin{mydef}
Let $\pi:X \rightarrow Y$ be a Borel map. A \textbf{system of measures} on $\pi$ is a family of measures $\lambda^{\bullet} = \{ \lambda^y \}_{y \in Y}$ such that:
\begin{enumerate}
\item Each $\lambda^y$ is a Borel measure on $X$;
\item For every $y$, $\lambda^y$ is concentrated on $\pi^{-1}(y)$.
\end{enumerate}
\end{mydef}
\noindent If the map $\pi:X \rightarrow Y$ is continuous (or proper if the spaces are $T_2$), then condition \emph{(2)} is equivalent to
\begin{enumerate}
\item[\emph{(2')}] \textit{For every} $y$, $supp(\lambda^y) \subseteq \pi^{-1}(y)$.
\end{enumerate}
This follows immediately from Lemma \ref{lem:support} since $\pi^{-1}(y)$ is a closed subset of $X$.

We will denote a map $\pi:X \rightarrow Y$ admitting a system of measures $\lambda^{\bullet}$ by the diagram
$\xymatrix{X\ar [rr]^{\pi}_{\lambda^{\bullet}}&&Y}$.

Trivially, when $Y$ is a singleton $\{y\}$, a system of measures on the projection $\pi:X \rightarrow \{y\}$ is merely a Borel measure on $X$. This obvious observation will be of use in the sequel.

\begin{mydef}
We will say that a system of measures $\lambda^{\bullet}$ is:
\begin{itemize}
\item
\textbf{positive on open sets} if $\lambda^y(A) > 0$ for every $y \in Y$ and for every open set $A \subseteq X$ such that $A \cap \pi^{-1}(y) \neq \emptyset$.
\item
\textbf{locally bounded} if for any $x \in X$ there exists a neighborhood $U_x$ and a constant $C >0$ such that $\lambda^y(U_x) < C$ for any $y \in Y$.
\end{itemize}
\end{mydef}

A system of measures will be called \emph{bounded on compact sets} if for any compact set $K \subseteq X$, $\lambda^{\bullet}(K)$ is a bounded function on $Y$.
In general, it is not hard to see that being locally bounded implies being bounded on compact sets. If $X$ is assumed to be locally compact, the converse is also trivially true. Our discussion of this property will usually be restricted to the setting of locally compact spaces, where the two notions coincide.

\begin{lem}\label{lem:positive iff full support}
Assume that the map $\pi:X\rightarrow Y$ is continuous. A system of measures $\lambda^{\bullet}$ on $\pi$ is positive on open sets if and only if $supp(\lambda^y) = \pi^{-1}(y)$ for every $y \in Y$.
\end{lem}

\begin{proof}
Suppose that $\lambda^{\bullet}$ is positive on open sets. For any $x\in \pi^{-1}(y)$ and any open neighborhood $A$ of $x$, we have that $A \cap \pi^{-1}(y) \neq \emptyset$ and thus $\lambda^{y}(A)>0$. Therefore, $x\in\ supp(\lambda^y)$.  This proves that $\pi^{-1}(y)\subseteq\ supp(\lambda^y)$. Condition \emph{(2')} above implies that $supp(\lambda^y) = \pi^{-1}(y)$.

Conversely, assume that $supp(\lambda^y) = \pi^{-1}(y)$ and let $A\subseteq X$ be an open subset satisfying $A\cap \pi^{-1}(y)\neq\emptyset$. Pick $x\in A\cap \pi^{-1}(y)$. Since $x\in supp(\lambda^y)$ and $A$ is an open neighborhood of $x$, it follows that $\lambda^y(A)>0$. Therefore, $\lambda^{\bullet}$ is positive on open sets.
\end{proof}

\begin{mydef}
A system of measures $\lambda^{\bullet}$ on a continuous map $\pi:X \rightarrow Y$ will be called
a \textbf{continuous system of measures} or \textbf{CSM} if for every non-negative continuous compactly supported function $0 \leq f\in C_c(X)$, the map $\displaystyle y \mapsto \int_X f(x)d \lambda^y (x)$ is a continuous function on $Y$.
\end{mydef}

Note that implicit in the above definition is the assumption on $\lambda^{\bullet}$ that $\int_X f(x)d \lambda^y (x)$ is finite for all $y$ and for any $0 \leq f\in C_c(X)$. This implies that $\int_X f(x)d \lambda^y (x)$ is finite for \textit{any} complex-valued function $f\in C_c(X)$. Hence, a CSM can be defined, equivalently, by requiring the map $y \mapsto \int_X f(x)d \lambda^y (x)$ to be a continuous function on $Y$ for any complex-valued function $f\in C_c(X)$.

\begin{mydef}
A system of measures $\lambda^{\bullet}$ on a Borel map $\pi:X \rightarrow Y$ is called
a \textbf{Borel system of measures} or \textbf{BSM} if for every Borel subset $E\subseteq X$, the function $\lambda^{\bullet}(E):Y\rightarrow [0,\infty]$ given by $y\mapsto \lambda^{y}(E)$ is a Borel function.
\end{mydef}

In the sequel it will be implicit that whenever a map $\pi:X \rightarrow Y$ admits a BSM, it is a Borel map, and if it admits a CSM, it is a continuous map. Also, recall that in the CSM context, spaces are assumed to be locally compact and Hausdorff.

\begin{lem}\label{lem:equiv_BSM}
A system of measures $\lambda^{\bullet}$ on $\pi:X\rightarrow Y$ is a BSM if and only if for every nonnegative Borel function $f:X \rightarrow [0,\infty]$, the map $\displaystyle y \mapsto \int f(x)d \lambda^y (x)$ is a Borel function on $Y$.
\end{lem}

\begin{proof}
Assume that $ y \mapsto \int f(x) d \lambda^y (x)$ is Borel for any Borel function $f:X \rightarrow [0,\infty]$, and let $E\subseteq X$ be a Borel subset. Then the function $ y \mapsto \int \chi_{_E}(x) d \lambda^y (x)= \lambda^y(E)$ is Borel.

Now suppose $\lambda^{\bullet}$ is a BSM. The following argument is standard. If $ s=\sum_{i=1}^n r_i\chi_{_{E_i}}$ is a nonnegative simple function on $X$, then the map $y \mapsto \int s(x) d \lambda^y (x)=\sum_{i=1}^n r_i\lambda^y(E_i)$ is Borel, being a linear combination of the Borel functions $y \mapsto \lambda^y(E_i)$. Now let $f$ be any nonnegative Borel function. There exists an increasing sequence of nonnegative simple functions $s_n$ that converges to $f$ pointwise on $X$. From the Monotone Convergence Theorem, $\int f(x)d \lambda^y (x)=\int \lim_{n\longrightarrow \infty}s_n(x)d \lambda^y (x)=\lim_{n \longrightarrow \infty}\int s_n(x)d \lambda^y (x)$. Therefore the function $y \mapsto \int f(x)d \lambda^y (x)$ is a limit of Borel functions and thus Borel.
\end{proof}

\begin{lem}
Let $\lambda^{\bullet}$ be a BSM. For any function $\displaystyle f\in \bigcap_{y\in Y}L^1(\lambda^y)$, the map $\displaystyle y \mapsto \int f(x)d \lambda^y (x)$ is Borel.
\end{lem}

\begin{proof}
The proof is a routine argument stemming from Lemma \ref{lem:equiv_BSM}. We will denote $F_f(y) = \int_X f(x)d \lambda^y (x)$. Assume first that $f$ is real-valued. Write $f=f_+-f_-$, where $f_+,f_-$ are respectively the positive and negative parts of $f$. By Lemma \ref{lem:equiv_BSM}, the functions $F_{f_+}(y)$ and $F_{f_-}(y)$
are both Borel and finite, which implies that the function $F_{f}(y) = F_{f_+}(y) - F_{f_-}(y)$
is Borel. For complex-valued $f$, write $f = f_1 + if_2$, and $F_{f}(y) = F_{f_1}(y) + iF_{f_2}(y)$ is Borel.
\end{proof}

\begin{lem}\label{lem:CSM_support}
Assume that $\lambda^{\bullet}$ is a CSM on $\pi:X \rightarrow Y$, and let $f \in C_c(X)$. Let $F:Y \rightarrow \mathbb{C}$ be the continuous function on $Y$ given by $\displaystyle F(y) = \int_X f(x)d \lambda^y (x)$. Then $supp(F) \subseteq \pi(supp(f))$.
\end{lem}

\begin{proof}
Define $A=\{x\in X:f(x)\neq 0\}$ and $B=\{y\in Y:F(y)\neq 0\}$. By definition, $\overline{A} = supp(f)$ and $\overline{B} = supp(F)$. Recall that $\lambda^y$ is concentrated on $\pi^{-1}(y)$, from which it follows that $$y \notin \pi(A) \ \Rightarrow \ \pi^{-1}(y) \cap A = \emptyset \ \Rightarrow \ \forall x \in \pi^{-1}(y), f(x)=0 \ \Rightarrow \ \int_X f(x)d \lambda^y (x) = 0 \ \Rightarrow \ y\notin B.$$ Thus $B\subseteq \pi(A)$. Since $\pi$ is continuous, $\overline{A}$ is compact, and $Y$ is $T_2$, we obtain $supp(F) = \overline{B} \subseteq \overline{\pi(A)} = \pi(\overline{A}) = \pi(supp(f))$.
\end{proof}

\begin{cor}\label{cor:CSM has compact support}
A CSM $\lambda^{\bullet}$ on $\pi:X \rightarrow Y$ satisfies that for every $f\in C_c(X)$, the map $\displaystyle y \mapsto \int_X f(x)d \lambda^y (x)$ is in $C_c(Y)$.
\end{cor}

In the literature, the compact support of the map $y \mapsto \int_X f(x)d \lambda^y (x)$ is often included in the definition of continuity for a system of measures.

\begin{lem}\label{lem:CSM always locally bounded}
A CSM is always locally bounded.
\end{lem}

\begin{proof}
Let $\lambda^{\bullet}$ be a continuous system of measures on the continuous map $\pi:X \rightarrow Y$ and let $K\subseteq X$ be compact. There exists a function $f \in C_c(X)$ such that $f:X\rightarrow [0, 1]$ and $f \equiv 1$ on $K$. Therefore, $\lambda^y(K) = \int_X \chi_{_K}(x) d\lambda^y(x) \leq \int_X f(x) d\lambda^y(x)$.
By Lemma \ref{lem:CSM_support}, the support of the continuous function $F(y) = \int_X f(x) d\lambda^y(x)$ is contained in $\pi(supp(f))$, which is compact. Therefore, $F$ is a bounded function on $Y$, and so is $\lambda^{\bullet}(K)$. Hence $\lambda^{\bullet}$ is bounded on compact sets and therefore locally bounded.
\end{proof}

\begin{mydef}
A system of measures $\lambda^{\bullet}$ on $\pi:X \rightarrow Y$ satisfying that $\lambda^{y}(X)<\infty$ for every $y\in Y$ will be called a \textbf{system of finite measures}. If $\lambda^{\bullet}$ is also a BSM, it will be called a \textbf{finite BSM}, and if $\lambda^{\bullet}$ is also a CSM, it will be called a \textbf{finite CSM}.
\end{mydef}

\begin{mydef}
A system of measures $\lambda^{\bullet}$ on $\pi:X \rightarrow Y$ satisfying that $\lambda^{y}(X)=1$ for every $y\in Y$ will be called a \textbf{system of probability measures}. If $\lambda^{\bullet}$ is also a BSM, it will be called a \textbf{probability BSM}, and if $\lambda^{\bullet}$ is also a CSM, it will be called a \textbf{probability CSM}.
\end{mydef}

\begin{mydef}
A system of measures $\lambda^{\bullet}$ on $\pi:X \rightarrow Y$ satisfying that every $x \in X$ has a neighborhood $U_x$ such that $\lambda^{y}(U_x)<\infty$ for every $y \in Y$, will be called a \textbf{locally finite system of measures}. If $\lambda^{\bullet}$ is also a BSM, it will be called a \textbf{locally finite BSM}.
\end{mydef}

A locally finite system of measures is, in particular, a system of locally finite measures. We deliberately chose the stronger notion, as it is needed for our purposes (in particular for Lemma \ref{lem:criterion for locally finite BSM}).

Observe that a system of measures which is locally bounded, is of course locally finite. In light of Lemma \ref{lem:CSM always locally bounded} we have the following immediate corollary.

\begin{cor}\label{cor:CSM is locally finite}
A CSM is always locally finite.
\end{cor}

Before we proceed, we briefly recall the following well known facts from basic measure theory. A \textbf{Dynkin system} $\mathcal D$ is a non-empty collection of subsets of a space $X$ which is
\begin{enumerate}[(i)]
\item closed under relative complements, i.e. if $A,B \in \mathcal D$ and $A \subseteq B$ then $B \setminus A \in \mathcal D$;
\item closed under countable unions of increasing sequences, i.e. if $A_i \in \mathcal D$ and $A_i \subseteq A_{i+1}$ then $\bigcup_{i=1}^{\infty} A_i \in \mathcal D$;
\item contains $X$ itself.
\end{enumerate}
An equivalent notion is that of a \textbf{$\mathbf{\lambda}$-system} $\mathcal D$, which is a non-empty collection of subsets of a space $X$ which is
\begin{enumerate}[(a)]
\item closed under complements, i.e. if $A \in \mathcal D$ then $A^c \in \mathcal D$;
\item closed under disjoint countable unions, i.e. if $A_i \in \mathcal D$ and $A_i \cap A_j = \emptyset \ \forall i \neq j$ then $\bigcup_{i=1}^{\infty} A_i \in \mathcal D$;
\item contains $X$ itself.
\end{enumerate}
A \textbf{$\mathbf{\pi}$-system} $\mathcal P$ is a non-empty collection of subsets that is closed under finite intersections. \textbf{Dynkin's $\mathbf{\pi}$-$\mathbf{\lambda}$ Theorem} says that if a $\pi$-system $P$ is contained in a Dynkin system $D$, then the entire $\sigma$-algebra generated by $\mathcal P$ is contained in $\mathcal D$.

For our purposes, the following definition will be useful.

\begin{mydef}\label{def:pre-dynkin}
We will say that a collection $\mathcal{D}$ of subsets of $X$ ia a \textbf{pre-Dynkin system} if it satisfies the following two properties:
\begin{enumerate}
\item if $E,F\mbox{ and }E\cap F\in{\mathcal D}$, then  $E\cup F$ and $E\setminus F \in {\mathcal D}$;
\item if ${\mathcal C} \subseteq {\mathcal D}$ is at most countable, and any finite intersection of elements in
${\mathcal C}$ belongs to $\mathcal D$, then the union of all elements of ${\mathcal C}$ belongs to $\mathcal D$.
\end{enumerate}
\end{mydef}

\begin{lem}\label{lem:equivalent Dynkin}
Let $\mathcal{D}$ be a collection of subsets of a space $X$. $\mathcal{D}$ is a Dynkin system if and only if $\mathcal{D}$ is a pre-Dynkin system and $X$ belongs to $\mathcal{D}$.
\end{lem}

\begin{proof}
Let $\mathcal{D}$ be a pre-Dynkin system on $X$ such that $X\in \mathcal{D}$. In order to prove that $\mathcal{D}$ is a Dynkin system, we verify properties (a), (b) and (c) above. Property (c) holds by assumption. For property (a), let $A \in \mathcal{D}$. Since $X\in \mathcal{D}$ and $X\cap A=A$, property (1) of a pre-Dynkin systems implies that $A^c=X\setminus A\in \mathcal{D}$, hence $\mathcal{D}$ is closed under complements. Finally, for property (b), let $\mathcal{C} = \{A_i\}_{i=1}^{\infty}\subseteq \mathcal{D}$ be a countable collection of pairwise disjoint subsets of $X$. For any finite intersection of distinct elements of $\mathcal{C}$ we have $$A_{i_1}\cap A_{i_2}\cap \dots\cap A_{i_k}=\left\{\begin{array}{lcl}A_{i_1}\in \mathcal{D}&\mbox{ if }&k=1,\\ \emptyset=X^c\in \mathcal{D}&\mbox{ if }&k>1.\end{array}\right.$$ Therefore, property (2) of a pre-Dynkin system guarantees that $\bigcup_{i=1}^{\infty} A_i \in \mathcal D$. We conclude that  $\mathcal{D}$ is a $\lambda$-system and thus a Dynkin system.

We now turn to the converse. Let $\mathcal{D}$ be a Dynkin system. Clearly, $X\in \mathcal{D}$. For property (1) of a pre-Dynkin system, let $E,\ F$ and $E\cap F\in \mathcal{D}$. Since by property (i) $\mathcal{D}$ is closed under relative complements, we have that $E\setminus F=E\setminus (E\cap F)\in \mathcal{D}$. Likewise, $F\setminus E=F\setminus (E\cap F)\in \mathcal{D}.$ From property (b) it follows that $\mathcal{D}$ is closed under disjoint finite unions, and thus we have that $E\cup F=(E\setminus F)\cup (F\setminus E)\cup (E\cap F)\in \mathcal{D}.$ For property (2) of a pre-Dynkin system, observe first that property (1) implies that if we have a \textit{finite} collection of sets in $\mathcal{D}$, satisfying that all their intersections are also in $\mathcal{D}$, then their union is in $\mathcal{D}$ as well. Now let $\mathcal{C}=\{C_i\}_{i=1}^{\infty}\subseteq \mathcal{D}$ be a \textit{countable} collection such that any finite intersection of its elements is in $\mathcal{D}$. Denote $V_k=\bigcup_{i=1}^k C_i$, for any $i\geq 1$. Applying the observation we just made to the finite collections $\mathcal{C}_k:=\{C_1,C_2,\dots, C_k\}$, we deduce that $V_k\in \mathcal{D}$, for all $k$. Since by property (ii) $\mathcal{D}$ is closed under countable unions of increasing sequences, we conclude that $\bigcup_{i=1}^{\infty}C_i=\bigcup_{k=1}^{\infty}V_k\in\mathcal{D}$. This completes the proof.
\end{proof}

\begin{prop}\label{prop:D contains Borel}
Let ${\mathcal D}$ be a pre-Dynkin system in $X$. If there is a countable basis ${\mathcal B}$ for the topology of $X$ such that $U_{1}\cap U_{2}\cap\dots\cap U_{n}\in {\mathcal D}$ for any $\{U_{1},U_{2},\dots, U_{n}\}\subset {\mathcal B}$, then ${\mathcal D}$ consists of all Borel subsets of $X$.
\end{prop}

\begin{proof}
Let $A$  be an open subset of $X$. Since ${\mathcal B}$ is a countable basis, there is a sequence $\{U_{1},U_{2},\dots, U_{n},\dots\}\subset {\mathcal B}$ such that $A=\bigcup_{i=1}^{\infty} U_i.$ Since, by assumption, all finite intersections of elements of the sequence belong to ${\mathcal D}$, property \textit{(2)} of Definition  \ref{def:pre-dynkin} implies that $A\in {\mathcal D}$. It follows that ${\mathcal D}$ contains all open subsets of $X$ and in particular  $X\in{\mathcal D}$. Therefore, ${\mathcal D}$ is a Dynkin system, containing all open subsets. Since open subsets form a $\pi$-system, we can invoke Dynkin's $\pi$-$\lambda$ Theorem to conclude that all Borel subsets of $X$ are in ${\mathcal D}$.
\end{proof}

\begin{lem}\label{lem:collection D}
Let $\lambda^{\bullet}$ be a system of finite measures. The collection of subsets
$${\mathcal D}=\{E\subseteq X\text{ Borel }:\ \lambda^{\bullet}(E)\text{ is a Borel function on } Y\}$$
is a pre-Dynkin system.
\end{lem}

\begin{proof}
We will prove that ${\mathcal D}$ satisfies properties \textit{(1)} and \textit{(2)} of Definition \ref{def:pre-dynkin}.
For any $y\in Y$ we have:
$${\lambda}^y(E\cup F)={\lambda}^y(E)+{\lambda}^y(F)-{\lambda}^y(E\cap F),\quad {\lambda}^y(E\setminus F)={\lambda}^y(E)-{\lambda}^y(E\cap F).$$ Therefore $\lambda^{\bullet}(E\cup F)$ and $\lambda^{\bullet}(E\setminus F)$ are Borel functions, and \textit{(1)} follows.

If ${\mathcal C}$ is finite, then \textit{(2)} is a consequence of an inclusion-exclusion formula as in \textit{(1)}. Suppose now that ${\mathcal C}$ is infinite, write ${\mathcal C}=\{E_n\}_{n=1}^{\infty}$ and let $E=\bigcup_{n=1}^{\infty} E_n$. Consider the sets
$$F_1:=E_1,\quad F_2:=E_1\cup E_2, \quad F_3:=E_1\cup E_2\cup E_3, \quad \dots$$
From the finite case we have that $F_n\in{\mathcal D}$ for all $n\geq 1$. Moreover, $E=\bigcup_{n=1}^{\infty} F_n$ and
$\lambda^{y}(E)=\lim_{n\rightarrow\infty} \lambda^{y}(F_n)$, for every $y\in Y$. Thus $\lambda^{\bullet}(E)$ is a Borel function, being a limit of the sequence of Borel functions $\{\lambda^{\bullet}(F_n)\}$. Therefore $E\in{\mathcal D}$, proving the infinite case of \textit{(2)}.
\end{proof}

\begin{lem}[Criterion for a system of finite measures to be a finite BSM]\label{lem:criterion for finite BSM}
Let $\pi:X\rightarrow Y$ be a Borel map endowed with a system of finite measures $\lambda^{\bullet}$. Assume that there is a countable basis ${\mathcal B}$ for the topology of $X$ such that $\lambda^{\bullet}(U_{1}\cap U_{2}\cap\dots\cap U_{n})$ is a Borel function for any $\{U_{1},U_{2},\dots, U_{n}\}\subset {\mathcal B}$, $n\geq 1$. Then $\lambda^{\bullet}$ is a finite BSM.
\end{lem}

\begin{proof}
Consider the collection ${\mathcal D}=\{E\subseteq X\text{ Borel }:\ \lambda^{\bullet}(E)\text{ is a Borel function on } Y\}$. By Lemma \ref{lem:collection D} above, ${\mathcal D}$ is a pre-Dynkin system. With respect to ${\mathcal D}$, the basis ${\mathcal B}$ satisfies the condition of Proposition \ref{prop:D contains Borel}, which in turn implies that all Borel subsets of $X$ are in ${\mathcal D}$. Therefore, $\lambda^{\bullet}$ is a BSM.
\end{proof}

\begin{lem}[Criterion for a locally finite system of measures to be a locally finite BSM]\label{lem:criterion for locally finite BSM}
Let $\pi:X\rightarrow Y$ be a Borel map endowed with a locally finite system of measures $\lambda^{\bullet}$. Assume that there is a countable basis ${\mathcal B}$ for the topology of $X$ such that $\lambda^{\bullet}(U_{1}\cap U_{2}\cap\dots\cap U_{n})$ is a Borel function for any $\{U_{1},U_{2},\dots, U_{n}\}\subset {\mathcal B}$, $n\geq 1$. Then $\lambda^{\bullet}$ is a locally finite BSM.
\end{lem}

\begin{proof}
Let ${\mathcal B}=\{U_i\}_{i=1}^{\infty}$. Since $\lambda^{\bullet}$ is locally finite, it is straightforward to verify that the sub-collection $\{U \in \mathcal{B} ~|~ \lambda^{y}(U)<\infty \text{ for every } y \in Y \}$ is itself a basis for $X$. Therefore, we can assume that all $U_i \in \mathcal{B}$ satisfy $\lambda^{y}(U_i)<\infty$ for every $y \in Y$.

For any $i \geq 1$, consider the map $\pi_i:U_i \rightarrow Y$ given by composing the inclusion $U_i\hookrightarrow X$ with $\pi : X\rightarrow Y$. Let
$\lambda^{\bullet}_i$ denote the system of measures on $\pi_i$ obtained by restricting $\lambda^{\bullet}$. Note that $\lambda^{\bullet}_i$ is a system of \textit{finite} measures, since $\lambda_i^{y}(U_i) = \lambda^{y}(U_i)<\infty$ for every $y \in Y$. Now consider the collection
$${\mathcal D}_i=\{E\subseteq U_i\text{ Borel }:\ \lambda_i^{\bullet}(E)\text{ is a Borel function on } Y\}.$$ We can apply Lemma \ref{lem:collection D}, which guarantees that ${\mathcal D}_i$ is a pre-Dynkin system in $U_i$. Also, the collection ${\mathcal B}_i=\left\{U_i\cap U_j\right\}_{j=1}^{\infty}$ is a basis for the topology of $U_i$. Moreover, due to our assumption on $\mathcal B$, we see that ${\mathcal B}_i$ satisfies the hypotheses of Proposition \ref{prop:D contains Borel} with respect to the collection ${\mathcal D}_i$. Consequently, ${\mathcal D}_i$ consists of all Borel subsets of $U_i$.

Let $E\subseteq X$ be a Borel subset. We need to show that $\lambda^{\bullet}(E)$ is a Borel function. For any $i$, the function $\lambda_i^{\bullet}(E\cap U_i)$ is a Borel function on $Y$, since $E\cap U_i$ is a Borel subset of $U_i$ and therefore in ${\mathcal D}_i$. Therefore, for any $i$, $\lambda^{\bullet}(E\cap U_i)$ is a Borel function on $Y$. Likewise, for any $i_1,\dots,i_k$ the function $\lambda^{\bullet}(E\cap U_{i_1}\cap\dots\cap U_{i_k})$ is a Borel function on $Y$.

Next, we define $V_n= \bigcup_{i=1}^n U_i$. This is an increasing sequence of open sets $\{V_n\}_{n=1}^{\infty}$, and each $V_n$ satisfies $\lambda^y(V_n) \leq \sum_{i=1}^n \lambda^y(U_i)<\infty$ for every $y \in Y$. Since $E\cap V_n = E \cap \left(\bigcup_{i=1}^n U_i \right) = (E \cap U_1) \cup (E \cap U_2) \cup \dots \cup (E \cap U_n)$, a routine inclusion-exclusion type argument yields that for all $n$, $\lambda^{\bullet}(E\cap V_n)$ can be written as a linear combination of functions of the form $\lambda^{\bullet}(E\cap U_{i_1}\cap\dots\cap U_{i_k})$, and is therefore a Borel function on $Y$.

Finally, $\lambda^{\bullet}(E)$ is the limit of the increasing sequence of Borel functions $\lambda^{\bullet}(E\cap V_n)$, hence by the Monotone Convergence Theorem, $\lambda^{\bullet}(E)$ is a Borel function on $Y$, as required.
\end{proof}

An immediate consequence of Lemma \ref{lem:criterion for locally finite BSM}
is the following.
\begin{cor}\label{cor:criterion for BSM}
Let $\pi:X\rightarrow Y$ be a Borel map endowed with a locally finite system of measures $\lambda^{\bullet}$.
If $\lambda^{\bullet}(A)$ is a Borel function for any open set $A$, then $\lambda^{\bullet}$ is a locally finite BSM.
\end{cor}

\begin{prop}\label{prop:CSM is BSM}
A CSM is a locally finite BSM.
\end{prop}

\begin{proof}
Let $\lambda^{\bullet}$ be a CSM on $\pi:X \rightarrow Y$. By Corollary \ref{cor:CSM is locally finite} $\lambda^{\bullet}$ is locally finite. By Corollary \ref{cor:criterion for BSM} it is sufficient to show that $\lambda^{\bullet}(A)$ is a Borel function for any open subset $A$.

There exists an increasing sequence $\{A_n\}_{n=1}^{\infty}$ of open subsets of $A$ such that $\overline{A_n}$ is compact for every $n$, $\overline{A_n}\subset A_{n+1}$ and $\bigcup_{n=1}^{\infty}A_n=A$. Moreover, there exists a non-decreasing sequence of compactly supported continuous functions $\psi_n:X\rightarrow [0, 1]$ such that $\psi_n\equiv 1$ on $A_n$ and $supp(\psi_n)\subseteq \overline{A_{n+1}}$ for every $n$. Therefore, $\forall x\in X$, $\lim_{n\rightarrow \infty}\psi_n(x)=\chi_{_A}(x)$. It follows, by the Monotone Convergence Theorem, that
$$\lambda^{y}(A)=\int_X\chi_{_A}(x)d\lambda^y(x)=\int_X\lim_{n\rightarrow \infty}\psi_n(x)d\lambda^y(x)=\lim_{n\rightarrow\infty}\int_X\psi_n(x)d\lambda^y(x).$$
Since $\forall n$, $\psi_n \in C_c(X)$ and $\lambda^{\bullet}$ is continuous, the map $y \mapsto \int_X\psi_n(x)d\lambda^y(x)$ is continuous $\forall n$. Therefore the map $y \mapsto \lambda^y(A)$ is a (monotone) limit of continuous (hence Borel) functions, and is thus a Borel function.
\end{proof}

We omit the full proof of the following lemma, which is analogous to the proof of Lemma \ref{lem:criterion for locally finite BSM}, via a corresponding version of Lemma \ref{lem:collection D} with ${\mathcal D}=\{E\subseteq X\text{ Borel }:\ \mu(E) = \nu(E) \}$.

\begin{lem}\label{lem:mu=nu for every E Borel}
Let $\mu$ and $\nu$ be two locally finite measures on a space $X$. Assume that there is a countable basis ${\mathcal B}$ for the topology of $X$ such that $\mu(U_{1}\cap U_{2}\cap\dots\cap U_{n})=\nu(U_{1}\cap U_{2}\cap\dots\cap U_{n})$ for any $\{U_{1},U_{2},\dots, U_{n}\}\subset {\mathcal B}$, $n\geq 1$. Then $\mu(E) = \nu(E)$ for any Borel subset $E \subseteq X$.
\end{lem}

\begin{cor}\label{cor:mu=nu for every A open}
Let $\mu$ and $\nu$ be two locally finite measures on a space $X$. If $\mu(A) = \nu(A)$ for any open subset $A \subseteq X$, then $\mu(E) = \nu(E)$ for any Borel subset $E \subseteq X$.
\end{cor}

We will make use of the above lemma and corollary in the sequel.

\section{Composition of systems of measures}\label{sec:composition}

The notion of composition of systems of measures appears in \S 1.3.a of \cite{renault-anantharaman-delaroche}, and is also mentioned briefly in \cite{revuz-book} (see Definition 1.5). Consider the diagram
$$\xymatrix{X\ar [rr]^{p}_{\alpha^{\bullet}}&&Y\ar [rr]^{q}_{\beta^{\bullet}}&&Z},$$
where $\alpha^{\bullet}$ is a BSM on $p:X \rightarrow Y$ and $\beta^{\bullet}$ is a system of measures  on $q:Y \rightarrow Z$.

\begin{mydef}\label{def:composition BSM}
We define the \textbf{composition} $(\beta\circ\alpha)^{\bullet}$ by $$(\beta\circ\alpha)^{z}(E)=\int_Y \alpha^{y}(E)\ d\beta^z(y) \ \ \ \ \ \forall z \in Z, \text{ and } E \subseteq X \text{ Borel. }$$
\end{mydef}

\begin{prop}\label{prop:composition BSM}
The composition $(\beta\circ\alpha)^{\bullet}$ is a system of measures on $q\circ p$. If $\alpha^{\bullet}$ and $\beta^{\bullet}$ are both BSMs, then $(\beta\circ\alpha)^{\bullet}$ is a BSM.
\end{prop}

\begin{proof}
Note that for any $z\in Z$ and any Borel subset $E \subseteq X$, $(\beta\circ\alpha)^{z}(E)$ is well defined, since $\alpha^{\bullet}(E)$ is a Borel function on $Y$ and $\beta^z$ is a Borel measure on $Y$. To prove that $(\beta\circ\alpha)^{z}$ is a Borel measure on $X$, let $\left\{E_n\right\}_{n=1}^{\infty}$ be a countable family of disjoint Borel subsets of $X$. Using a standard Monotone Convergence Theorem argument with $\sum_{n=1}^{k}\alpha^{y}\left( E_n \right) \nearrow  \sum_{n=1}^{\infty}\alpha^{y}\left( E_n\right)$, we obtain
\begin{eqnarray}
(\beta\circ\alpha)^{z}(\bigcup_{n=1}^{\infty} E_n) &=& \int_Y \alpha^{y}(\bigcup_{n=1}^{\infty} E_n)\ d\beta^z(y)=\int_Y \sum_{n=1}^{\infty}\alpha^{y}\left( E_n\right)\ d\beta^z(y)= \nonumber \\ &=& \sum_{n=1}^{\infty}\int_Y \alpha^{y}\left( E_n\right)\ d\beta^z(y) = \sum_{n=1}^{\infty}(\beta\circ\alpha)^{z}(E_n).\nonumber
\end{eqnarray}
To prove that $(\beta\circ\alpha)^{z}$ is concentrated on $(q\circ p)^{-1}(z)$, observe that
if $y\in q^{-1}(z)$ then $p^{-1}(y) \subseteq (q\circ p)^{-1}(z)$. Taking complements in $X$ we get $\alpha^y(X\setminus (q\circ p)^{-1}(z)) \leq \alpha^y(X\setminus p^{-1}(y)) = 0$. Since $\beta^z$ is concentrated on $q^{-1}(z)$, we obtain $$(\beta\circ\alpha)^{z}\left(X\setminus (q\circ p)^{-1}(z)\right)=\int_Y\alpha^y\left(X\setminus (q\circ p)^{-1}(z)\right)d\beta^z(y)=0.$$
We have shown that $(\beta\circ\alpha)^{\bullet}$ is a system of measures on $q \circ p$. Now assume that both $\alpha^{\bullet}$ and $\beta^{\bullet}$ are BSMs. Let $E\subseteq X$ be a Borel subset. Since $\alpha^{\bullet}$ is a BSM, the function $\alpha^{\bullet}(E)$ is a nonnegative Borel function on $Y$. But $\beta^{\bullet}$ is a BSM as well, so from Lemma \ref{lem:equiv_BSM} we have that $z\mapsto \int_Y \alpha^{y}(E)\ d\beta^z(y)$, which is precisely the function $(\beta\circ\alpha)^{\bullet}(E)$, is a Borel function on $Z$. Therefore, $(\beta\circ\alpha)^{\bullet}$ is a BSM. This completes the proof.
\end{proof}

\begin{prop}\label{prop:composition CSM}
If $\alpha^{\bullet}$ and $\beta^{\bullet}$ are both CSMs, then $(\beta\circ\alpha)^{\bullet}$ is a CSM.
\end{prop}

\begin{proof}
Let $f\in C_c(X)$. We need to show that the map $z \mapsto \int_X f(x)d (\beta \circ \alpha)^z (x)$ is a continuous function on $Z$. Define $g(y) = \int_X f(x)d\alpha^y (x)$. Since $\alpha^{\bullet}$ is a CSM, Corollary \ref{cor:CSM has compact support} implies that $g(y) \in C_c(Y)$. From the fact that $\beta^{\bullet}$ is a CSM we now get that the map $z \mapsto \int_Y g(y) d\beta^z(y) \in C_c(Z)$. This completes the proof, since
$$\int_Y g(y) d\beta^z(y) = \int_Y \left(\int_X f(x)d\alpha^y (x) \right) d\beta^z(y) = \int_X f(x)d (\beta \circ \alpha)^z (x).$$
\end{proof}

\begin{prop}\label{prop:composition pos and bdd}
Consider the setting of Definition \ref{def:composition BSM}.
\begin{enumerate}
\item
Assume that $p$ is an open map. If $\alpha^{\bullet}$ and $\beta^{\bullet}$ are positive on open sets then so is $(\beta\circ\alpha)^{\bullet}$.
\item
Assume that $p$ is a continuous map. If $\alpha^{\bullet}$ and $\beta^{\bullet}$ are locally bounded then so is $(\beta\circ\alpha)^{\bullet}$.
\end{enumerate}
\end{prop}

\begin{proof}
(1) Fix $z \in Z$ and let $A\subseteq X$ be an open set satisfying $A\cap (q\circ p)^{-1}(z)\neq \emptyset$. We need to show that $(\beta\circ\alpha)^{z}(A) >0$. The set $p(A)\cap q^{-1}(z)$ is not empty since $$\emptyset \neq p\left(A\cap (q\circ p)^{-1}(z)\right) \subseteq p(A)\cap p(p^{-1}(q^{-1}(z))) \subseteq p(A)\cap q^{-1}(z).$$ Furthermore, $p$ is assumed to be an open map, so $p(A)$ is open in $Y$. This implies that $\beta^z(p(A))>0$, since $\beta^{\bullet}$ is positive on open sets. Obviously, for every $y\in p(A)$ there exists $x\in A$ such that $p(x)=y$, hence $p^{-1}(y)\cap A\neq\emptyset$. This implies that $\alpha^y(A)>0$ for every $y\in p(A)$, since $\alpha^{\bullet}$ is positive on open sets. We conclude that $$(\beta\circ\alpha)^{z}(A)=\int_Y \alpha^{y}(A)\ d\beta^z(y)\geq \int_{p(A)} \alpha^{y}(A)\ d\beta^z(y)>0.$$
(2) Take $x \in X$. Since $\beta^{\bullet}$ is locally bounded, there exists an open neighborhood $V$ of $p(x)$ and a constant $C_2$ such that $\beta^z(V) < C_2$ for every $z \in Z$. Since $p$ is continuous and $\alpha^{\bullet}$ is locally bounded, there exists an open neighborhood $U$ of $x$ and a constant $C_1$ such that $p(U) \subseteq V$ and $\alpha^y(U) < C_1$ for every $y \in Y$. Note that if $y \notin p(U)$, then $p^{-1}(y)\cap U=\emptyset$. Hence $\alpha^y(U)=0$ for all $y \notin p(U)$. We therefore have $$(\beta\circ\alpha)^{z}(U)=\int_Y \alpha^{y}(U)\ d\beta^z(y)=\int_{p(U)} \alpha^{y}(U)\ d\beta^z(y)\leq C_1\cdot\beta^z(p(U))\leq C_1\cdot C_2$$ for every $z\in Z$.
\end{proof}

In general, the composition of locally finite systems of measures need not be locally finite. In order to assure local finiteness of the composition we need to require a stronger property of the system $\alpha^{\bullet}$. We omit the proof of the following Lemma, which is an obvious modification of the proof of the second part of Proposition \ref{prop:composition pos and bdd}.

\begin{lem}\label{lem:condition for composition locally finite}
Consider the setting of Definition \ref{def:composition BSM}, and assume that the map $p$ is continuous. If $\alpha^{\bullet}$ is locally bounded and $\beta^{\bullet}$ is locally finite then $(\beta\circ\alpha)^{\bullet}$ is locally finite.
\end{lem}

We have seen in Lemma \ref{lem:CSM always locally bounded} that any CSM is locally bounded. Taken together with Lemma \ref{lem:condition for composition locally finite} and the fact that a locally bounded system is in particular locally finite, this implies that the composition is guaranteed to be locally finite in several more scenarios.

\begin{cor}\label{cor:conditions for composition locally finite}
Consider the setting of Definition \ref{def:composition BSM}, and assume that the map $p$ is continuous. Each of the following conditions implies that $(\beta\circ\alpha)^{\bullet}$ is locally finite.
\begin{enumerate}
\item
$\alpha^{\bullet}$ is a CSM and $\beta^{\bullet}$ is locally finite.
\item
$\alpha^{\bullet}$ and $\beta^{\bullet}$ are both locally bounded.
\item
$\alpha^{\bullet}$ and $\beta^{\bullet}$ are both CSMs.
\item
either $\alpha^{\bullet}$ or $\beta^{\bullet}$ is a CSM and the other is locally bounded.
\end{enumerate}
\end{cor}

As a particular case of Lemma \ref{lem:condition for composition locally finite} we obtain the following useful result. The proof amounts to taking $Z = \{z\}$, viewing $\beta$ as a trivial system of measures on the projection $\pi:Y \rightarrow \{z\}$ and applying Lemma \ref{lem:condition for composition locally finite}.

\begin{cor}\label{cor:locally finite composition of BSM and measure}
Let $\alpha^{\bullet}$ be a locally bounded BSM on a continuous map $p:X \rightarrow Y$ and let $\beta$ be a locally finite measure on $Y$. For every Borel set $E \subseteq X$, define $$\mu(E) = \int_{Y} \alpha^y (E) d\beta (y) .$$ Then $\mu$ is a locally finite measure on $X$.
\end{cor}

\section{Lifting of systems of measures}\label{sec:lifting}

The concept of lifting, which we define below, is discussed in Appendix A.1 of \cite{renault-anantharaman-delaroche}, in the broader context of transverse measure theory.
Let $X$, $Y$ and $Z$ be topological spaces, and let $p:X \rightarrow Z$ and $q:Y \rightarrow Z$ be Borel maps. The usual \textbf{pullback} of $X$ and $Y$ over $Z$  is the space $$X *_Z Y = \{(x,y) \in X \times Y~:~p(x) = q(y) \}.$$ In order to lighten notation, we will usually write $X * Y$, keeping $Z$ implicit. The topology on $X * Y$ is inherited from the product topology on $X \times Y$. Consider the pullback diagram
$$\xymatrix{X * Y \ar [dd]_{\pi_X}\ar [rr]^{\pi_Y}&&Y\ar [dd]_{q}\\\\
X\ar [rr]^{p}_{\alpha^{\bullet}}&&Z}$$ where $\pi_X$ and $\pi_Y$ are the obvious projections, and $\alpha^{\bullet}$ is a system of measures on $p:X \rightarrow Z$. Observe that the fibers of the map $\pi_Y$ are Cartesian products of the form $\pi_Y^{-1}(y)=p^{-1}(q(y))\times\{y\}$. \\

We will assume throughout this section that $\alpha^{\bullet}$ is a \textit{locally finite} system of measures.

\begin{mydef}\label{def:lifting BSM}
The \textbf{lifting} of the locally finite system of measures $\alpha^{\bullet}$ to $\pi_Y$, denoted $(q^*\alpha)^{\bullet}$, is given by $$(q^*\alpha)^y=\alpha^{q(y)}\times \delta_y.$$ More precisely, $(q^*\alpha)^{y}(E)=(\alpha^{q(y)}\times \delta_y)(E\cap \pi_Y^{-1}(y))$ for every $y\in Y$ and every Borel set $E\subseteq X*Y$.
\end{mydef}

\begin{rem}
If $\beta^{\bullet}$ is a locally finite system of measures on $q:Y \rightarrow Z$, then the lifting $(p^*\beta)^{\bullet}$ to $\pi_X$ is defined similarly, by $(p^*\beta)^x=\delta_x\times \beta^{p(x)}.$ The properties of the lifting $(q^*\alpha)^{\bullet}$ which we state and prove below, hold for $(p^*\beta)^{\bullet}$ as well, with the obvious modifications.
\end{rem}

\begin{rem}\label{rem:elementary open sets}
In the sequel, we will make frequent use of open sets $E \subseteq X*Y$ of the form $E=(A\times B)\cap (X*Y)$, where $A$ and $B$ are open sets in $X$ and $Y$ respectively. We will refer to these as \textit{elementary open sets}. For any elementary open set we have
\begin{eqnarray}
(q^*\alpha)^{y}(E)&=&(q^*\alpha)^{y}(E\cap \pi_Y^{-1}(y))\label{calc:lifting}\nonumber\\
                  &=&(q^*\alpha)^{y}\left((A\cap p^{-1}(q(y)))\times (B\cap\{y\})\right)\\
                  &=&\alpha^{q(y)}(A\cap p^{-1}(q(y)))\cdot \delta_y(B\cap\{y\})\nonumber\\
                  &=&\alpha^{q(y)}(A)\cdot \delta_y(B).\nonumber
\end{eqnarray}
If $\{A_n\}_{n=1}^{\infty}$ and $\{B_m\}_{m=1}^{\infty}$ are countable bases for the topologies of $X$ and $Y$ respectively, we can set ${\mathcal B}=\{(A_n\times B_m)\cap X*Y\}_{n,m=1}^{\infty}$. This gives a countable basis ${\mathcal B}$ for the topology of $X*Y$ consisting of elementary open sets.
\end{rem}

\begin{prop}\label{prop:lifting BSM}
The lifting $(q^*\alpha)^{\bullet}$ is a locally finite system of measures on $\pi_Y$. If $\alpha^{\bullet}$ is a BSM, then so is $(q^*\alpha)^{\bullet}$.
\end{prop}

\begin{proof}
As a product of locally finite (hence $\sigma$-finite) Borel measures, $(q^*\alpha)^{y}$ is a well defined Borel measure for every $y\in Y$. By definition it is concentrated on $p^{-1}(q(y))\times\{y\} = \pi_Y^{-1}(y)$.
Let $(x,y) \in X*Y$. Since $\alpha^{\bullet}$ is locally finite, there exists a neighborhood $U_x$ of $x$ such that $\alpha^z(U_x) < \infty$ for all $z \in Z$. By calculation (\ref{calc:lifting}) above, the open neighborhood $(U_x \times Y) \cap (X*Y)$ of $(x,y)$ satisfies $(q^*\alpha)^y\left((U_x \times Y) \cap (X*Y)\right) = \alpha^{q(y)}(U_x)\cdot \delta_y(Y) = \alpha^{q(y)}(U_x) < \infty$ for every $y \in Y$, hence $(q^*\alpha)^{\bullet}$ is a locally finite system of measures.

Now assume that $\alpha^{\bullet}$ is a BSM. In order to prove that $(q^*\alpha)^{\bullet}$ is a BSM, we show first that $(q^*\alpha)^{\bullet}(E)$ is a Borel function for any elementary open set $E=(A\times B)\cap (X*Y)$. For such $E$ we have, by calculation (\ref{calc:lifting}), that $(q^*\alpha)^{y}(E) = \alpha^{q(y)}(A)\cdot \delta_y(B).$ Therefore, if we denote the composition of the Borel functions $\alpha^{\bullet}(A)$ and $q$ by $\alpha^{q(\bullet)}(A)$, we can write $(q^*\alpha)^{\bullet}(E)=\alpha^{q(\bullet)}(A)\cdot \chi_{_B}$. Thus $(q^*\alpha)^{\bullet}(E)$ is a Borel function.

Finite intersections of elementary open sets are themselves elementary open sets, and thus the basis ${\mathcal B}$ as in Remark \ref{rem:elementary open sets} satisfies the hypotheses of Lemma \ref{lem:criterion for locally finite BSM}. We conclude that $(q^*\alpha)^{\bullet}$ is a BSM.
\end{proof}

\begin{lem}\label{lem:continuity lemma}
Let $\mathcal{X}$, $\mathcal{Y}$ and $\mathcal{Z}$ be topological spaces and let $\gamma^{\bullet}$ be a CSM on $\phi:\mathcal{Y}\rightarrow \mathcal{Z}$. For every $\psi\in C_c(\mathcal{X}\times \mathcal{Y})$, the function $\displaystyle (x,z)\mapsto \int_{\mathcal{Y}} \psi(x,y)d\gamma^z(y)$ belongs to $C_c(\mathcal{X}\times \mathcal{Z})$.
\end{lem}

\begin{proof}
We first show that $F(x,z)= \int_{\mathcal{Y}} \psi(x,y)d\gamma^z(y)$ has compact support. Let $\pi_{\mathcal{X}}: \mathcal{X}\times \mathcal{Y} \rightarrow \mathcal{X}$ and $\pi_{\mathcal{Y}}: \mathcal{X}\times \mathcal{Y} \rightarrow \mathcal{Y}$ denote the projections, and let $K \subseteq \mathcal{X}\times \mathcal{Y}$ be the support of $\psi$. Observe that if $(x,z)\notin \pi_{\mathcal{X}}(K)\times \phi(\pi_{\mathcal{Y}}(K))$, then $(x,y)$ does not belong to $K$ for any $y\in \phi^{-1}(z)$.
Therefore, for such $(x,z)$ we have $F(x,z) = \int_{\mathcal{Y}} \psi(x,y)d\gamma^z(y)=\int_{\phi^{-1}(z)} \psi(x,y)d\gamma^z(y)=0$. Thus $\{(x,z) ~|~ F(x,z) \neq 0\}$ is contained in $\pi_{\mathcal{X}}(K)\times \phi(\pi_{\mathcal{Y}}(K))$ which is compact (hence closed), and it follows that $supp(F) \subseteq \pi_{\mathcal{X}}(K)\times \phi(\pi_{\mathcal{Y}}(K))$ is compact.

We turn to proving that $F$ is continuous on $\mathcal{X}\times \mathcal{Z}$. Fix $x_0\in \mathcal{X}$, $z_0\in \mathcal{Z}$ and $\epsilon>0$. We claim that there exists a neighborhood $A_{x_0}$ of $x_0$ such that $\sup_{y}\left|\psi(x,y)-\psi(x_0,y)\right|<2\epsilon$ for any $x \in A_{x_0}$.

Let $ y' \in \pi_{\mathcal{Y}}(K)$. Since $\psi$ is continuous, there exist open sets $A_{x_0,y'}\subset \mathcal{X}$ and $B_{x_0,y'}\subset\mathcal{Y}$ such that $(x_0, y')\in A_{x_0,y'}\times B_{x_0,y'}$, and $\left|\psi(x,y)-\psi(x_0,y')\right|<\epsilon$ for any $(x, y)\in A_{x_0,y'}\times B_{x_0,y'}$. In particular, $\left|\psi(x,y)-\psi(x_0,y)\right|\leq \left|\psi(x,y)-\psi(x_0,y')\right|+\left|\psi(x_0,y')-\psi(x_0,y)\right|<2 \epsilon$. Since $\{x_0\}\times \pi_{\mathcal{Y}}(K)$ is compact, it admits a finite cover $\bigcup_{i=1}^n \left(A_{x_0,y_i'}\times B_{x_0,y_i'} \right)$. Define $A_{x_0}=\bigcap_{i=1}^n A_{x_0,y_i'}$ and $B_{x_0}=\bigcup_{i=1}^n B_{x_0,y_i'}$. Now consider $(x,y)\in A_{x_0}\times \mathcal{Y}$. If $y\in B_{x_0}$, then $\left|\psi(x,y)-\psi(x_0,y)\right|<2\epsilon$. If $y\notin B_{x_0}$ then $(x,y), (x_0, y)\notin K$, hence $\left|\psi(x,y)-\psi(x_0,y)\right|=|0-0|=0$. Thus, for any $x\in A_{x_0}$, we have
$\sup_{y}\left|\psi(x,y)-\psi(x_0,y)\right|<2\epsilon$, as claimed.

For every $x\in A_{x_0}$ and $z\in\mathcal{Z}$,
\begin{eqnarray*}
\left| F(x,z) - F(x_0,z) \right| &=& \left|\int_{\mathcal{Y}} \psi(x,y)d\gamma^z(y)-\int_{\mathcal{Y}} \psi(x_0,y)d\gamma^z(y)\right| \\ &=& \left|\int_{\mathcal{Y}} \left( \psi(x,y)-\psi(x_0,y) \right) d\gamma^z(y)\right| \\ &\leq&\int_{\mathcal{Y}} \left| \psi(x,y)-\psi(x_0,y)\right|d\gamma^z(y) \ < \ 2 \; \epsilon \; \gamma^z(\pi_{\mathcal{Y}}(K))
\end{eqnarray*}
Since $\gamma^{\bullet}$ is a CSM, by Lemma \ref{lem:CSM always locally bounded} it is locally bounded, or equivalently - bounded on compact sets. It follows that for every $x\in A_{x_0}$ and $z\in\mathcal{Z}$,
$$\left| F(x,z) - F(x_0,z) \right| \ < \ \epsilon \cdot C, $$
where $C$ is a constant depending only on $K$ and $\gamma^{\bullet}$.

On the other hand, by the definition of a CSM, there is a neighborhood $V_{z_0}$ of $z_0$ such that for any $z\in V_{z_0}$
$$\left| F(x_0,z) - F(x_0,z_0) \right| \ < \ \epsilon.$$

We conclude that for every $(x,z)\in A_{x_0}\times V_{z_0}$,
$$
\left| F(x,z) - F(x_0,z_0) \right| \leq \left| F(x,z) - F(x_0,z) \right| + \left| F(x_0,z) - F(x_0,z_0) \right| \ < \ \epsilon \ (C+1),
$$
hence $F$ is continuous.
\end{proof}

\begin{prop}\label{prop:lifting CSM}
If $\alpha^{\bullet}$ is a CSM, then so is the lifting $(q^*\alpha)^{\bullet}$.
\end{prop}

\begin{proof}
Let $f \in  C_c(X * Y)$. We need to show that the function $y \mapsto \int_{X * Y} f(x,\eta) d(q^*\alpha)^y(x,\eta)$ is continuous on $Y$. The space $X * Y$ is closed in $X \times Y$, as the inverse image of the diagonal $\Delta(Z)$ under the continuous map $(p,q)$. Therefore, by Tietze's Extension Theorem, there exists a function $F \in C(X \times Y)$ such that $F|_{X * Y} = f$. Since we can multiply $F$ by a function $\varphi \in C_c(X \times Y)$ which satisfies $\varphi =1$ on $K = supp(f)$, we can assume, without loss of generality, that $F \in C_c(X \times Y)$.

We now apply (a symmetric version of) lemma \ref{lem:continuity lemma} above, and obtain that the map $(y,z) \mapsto \int_{X \times Y} F(x,y)d\alpha^z(x)$ belongs to $C_c(Y \times Z)$. Composing with the continuous function $y \mapsto (y,q(y))$, we deduce that the map $y \mapsto \int_{X \times Y} F(x,y)d\alpha^{q(y)}(x)$ is continuous on $Y$.

Observe that $\alpha^{q(y)}$ is concentrated on $p^{-1}(q(y))$. Therefore, since $ p^{-1}(q(y))\times\{y\} \subset X*Y$, we have
\begin{eqnarray*}
\int_{X \times Y} F(x,y)d\alpha^{q(y)}(x) &=& \int_{p^{-1}(q(y))\times\{y\}} F(x,y)d\alpha^{q(y)}(x) \ =  \ \int_{X*Y} f(x,y)d\alpha^{q(y)}(x) \\ &=& \int_{X*Y} f(x,\eta) d(\alpha^{q(y)} \times \delta_y)(x,\eta) \ = \  \int_{X*Y} f(x,\eta) d(q^*\alpha)^y(x,\eta).
\end{eqnarray*}
We conclude that the map $y \mapsto \int_{X*Y} f(\xi,\eta) d(q^*\alpha)^y(\xi,\eta)$ is continuous on $Y$, as required.
\end{proof}

\begin{prop}\label{prop:lifting stability}
The properties of being positive on open sets and locally bounded are preserved under lifting.
\end{prop}

\begin{proof}
Assume that $\alpha^{\bullet}$ is positive on open sets. In order to prove that $(q^*\alpha)^{\bullet}$ is positive on open sets, it suffices to consider only elementary open sets, since they generate the topology of $X*Y$. So fix $y \in Y$ and let $E=(A\times B)\cap (X*Y)$ be an elementary open set such that $E\cap\pi_Y^{-1}(y)\neq \emptyset$. This implies that $A\cap p^{-1}(q(y))\neq \emptyset$ and $y\in B$, hence $\alpha^{q(y)}(A)>0$ and $\delta_y(B)=1$. Using calculation (\ref{calc:lifting}) above we obtain that $(q^*\alpha)^{y}(E)=\alpha^{q(y)}(A)\cdot \delta_y(B)>0$. This proves that $(q^*\alpha)^{\bullet}$ is positive on open sets.

Proving that the lifted system is locally bounded is similar to the proof that it is locally finite in Proposition \ref{prop:lifting BSM}.
\end{proof}

Consider the pull-back diagram
$$\xymatrix{X * Y \ar [dd]_{\pi_X}^{(p^*\beta)^{\bullet}}\ar [rr]^{\pi_Y}_{(q^*\alpha)^{\bullet}}&&Y\ar [dd]_{q}^{\beta^{\bullet}}\\\\ X\ar [rr]^{p}_{\alpha^{\bullet}}&&Z}$$
where $\beta^{\bullet}$ and  $\alpha^{\bullet}$ are locally finite BSMs, and $(p^*\beta)^{\bullet}$ and $(q^*\alpha)^{\bullet}$ are their lifting to $\pi_X$ and $\pi_Y$, respectively. By Proposition \ref{prop:lifting BSM}, $(p^*\beta)^{\bullet}$ and $(q^*\alpha)^{\bullet}$ are also locally finite BSMs.

\begin{prop}
The above pull-back diagram is a commutative diagram of locally finite BSMs. In other words, $(\beta\circ q^*\alpha)^{\bullet}$ and $(\alpha\circ p^*\beta)^{\bullet}$ are locally finite and $$(\beta\circ q^*\alpha)^{\bullet}=(\alpha\circ p^*\beta)^{\bullet}.$$
\end{prop}

\begin{proof}
Fix $z\in Z$ and denote $\mu = (\beta\circ q^*\alpha)^{z}$ and $\nu = (\alpha\circ q^*\beta)^{z}$.
We claim that $\mu(E) = \nu(E)$ for any elementary open subset of $X*Y$. Indeed, let $E=(A\times B)\cap (X*Y)$. Then by calculation (\ref{calc:lifting}) preceding Proposition \ref{prop:lifting BSM}, and recalling that $\beta^z$ is concentrated on $q^{-1}(z)$, we have:
\begin{eqnarray}
\mu(E)=(\beta\circ q^*\alpha)^{z}(E)&=&\int_{Y}(q^*\alpha)^y(E)d\beta^z(y) \ = \ \int_{Y}\alpha^{q(y)}(A)\cdot\delta_y(B)d\beta^z(y) \ = \nonumber \\
&=&\int_{B\cap q^{-1}(z)}\alpha^{q(y)}(A)d\beta^z(y) \ = \ \int_{B}\alpha^{z}(A)d\beta^z(y) \ = \ \alpha^{z}(A)\beta^z(B). \nonumber
\end{eqnarray}
Analogously,
\begin{eqnarray}
\nu(E)=(\alpha\circ p^*\beta)^{z}(E) & = & \int_{X}(p^*\beta)^x(E)d\alpha^z(x) \
= \ \int_{X}\delta_x(A)\cdot\beta^{p(x)}(B)d\alpha^z(x) \ = \nonumber \\
&=& \int_{A \cap p^{-1}(z)}\beta^{p(x)}(B)d\alpha^z(x) \ = \ \int_{A}\beta^{z}(B)d\alpha^z(x) \ = \ \alpha^{z}(A)\beta^z(B). \nonumber
\end{eqnarray}
Therefore $\mu(E) = \nu(E)$ for any elementary open set.

The systems $\alpha^{\bullet}$ and $\beta^{\bullet}$ are locally finite. Thus, the topology of $X*Y$ admits a basis $\mathcal{B}$ as in Remark \ref{rem:elementary open sets}, comprised of elementary open sets of the form $E=(A\times B)\cap (X*Y)$ satisfying that $\alpha^{z}(A)$ and $\beta^z(B)$ are both finite, for any $z \in Z$. It follows from the above calculations that the compositions $(\beta\circ q^*\alpha)^{\bullet}$ and $(\alpha\circ p^*\beta)^{\bullet}$ are locally finite systems, and moreover, $\mu$ and $\nu$ are locally finite measures. Since finite intersections of elementary open sets are themselves elementary open sets,
we can apply Lemma \ref{lem:mu=nu for every E Borel} with the basis ${\mathcal B}$, and conclude that $\mu(E) = \nu(E)$ for every Borel subset $E \subseteq X$. This completes the proof.
\end{proof}

\section{Fibred products of systems of measures}\label{sec:fibred products}

Fibred products are mentioned in \S 1.3.a in \cite{renault-anantharaman-delaroche}. Assume that we have two pullback diagrams:
$$\xymatrix{X_i * Y_i \ar [dd]_{\pi_{X_i}}\ar [rr]^{\pi_{Y_i}}&&Y_i\ar [dd]_{q_i}\\\\
X_i\ar [rr]^{p_i}&&Z}$$ Also, let $\xymatrix{X_1\ar [rr]^{f}_{\gamma_X^{\bullet}}&&X_2}$ and $\xymatrix{Y_1\ar [rr]^{g}_{\gamma_Y^{\bullet}}&&Y_2}$ be connecting maps endowed with locally finite systems of measures, satisfying that $p_1=p_2\circ f$ and $q_1=q_2\circ g$. Putting these together we obtain the following diagram:
\[
\xy 0;<.2cm,0cm>:
(20,20)*{X_2 * Y_2}="1";
(40,20)*{Y_2}="2";
(20,0)*{X_2}="3";
(40,0)*{Z}="4";
(10,10)*{X_1 * Y_1}="5";
(30,10)*{Y_1}="6";
(10,-10)*{X_1}="7";
(30,-10)*{Z}="8";
{"1"+CR+(.5,0);"2"+CL+(-.5,0)**@{-}?>*{\dir{>}}?>(.5)+(0,1)*{\scriptstyle \pi_{Y_2}}};
{"5"+CR+(.5,0);"6"+CL+(-.5,0)**@{-}?>*{\dir{>}}?>(.75)+(0,1)*{\scriptstyle \pi_{Y_1}}};
{"3"+CR+(.5,0);"4"+CL+(-10,0)**@{-}};
{"3"+CR+(9,0);"4"+CL+(-.5,0)**@{-}?>*{\dir{>}}?>(.5)+(0,1)*{\scriptstyle p_2}};
{"7"+CR+(.5,0);"8"+CL+(-.5,0)**@{-}?>*{\dir{>}}?>(.5)+(0,1)*{\scriptstyle p_1}};
{"1"+CD+(0,-.5);"3"+CU+(0,9.6)**@{-}} ;
{"1"+CD+(0,-9.8);"3"+CU+(0,.5)**@{-}?>*{\dir{>}}?>(.5)+(-1.5,0)*{\scriptstyle \pi_{X_2}}} ;
{"2"+CD+(0,-.5);"4"+CU+(0,.5)**@{-}?>*{\dir{>}}?>(.5)+(-1,0)*{\scriptstyle q_2}} ;
{"5"+CD+(0,-.5);"7"+CU+(0,.5)**@{-}?>*{\dir{>}}?>(.5)+(-1.5,0)*{\scriptstyle \pi_{X_1}}} ;
{"6"+CD+(0,-.5);"8"+CU+(0,.5)**@{-}?>*{\dir{>}}?>(.25)+(-1,0)*{\scriptstyle q_1}} ;
{"5"+C+(1.4,1.4);"1"+C+(-1.4,-1.4)**@{-}?>*{\dir{>}}?>*{\dir{>}}?>(.5)+(-2,0)*{\scriptstyle f*g}} ;
{"6"+C+(1.4,1.4);"2"+C+(-1.4,-1.4)**@{-}?>*{\dir{>}}?>(.5)+(-1,0)*{\scriptstyle g}?>(.5)+(1.5,0)*{\scriptstyle \gamma_Y^{\bullet}}} ;
{"7"+C+(1.4,1.4);"3"+C+(-1.4,-1.4)**@{-}?>*{\dir{>}}?>(.5)+(-1,0)*{\scriptstyle f}?>(.5)+(1.5,0)*{\scriptstyle \gamma_X^{\bullet}}} ;
{"8"+C+(1.4,1.4);"4"+C+(-1.4,-1.4)**@{-}?>*{\dir{>}}?>(.5)+(-1.5,0)*{\scriptstyle id}};
\endxy
\]
where the map $f*g=X_1*Y_1\rightarrow X_2*Y_2$ is defined by $(f*g)(x_1,y_1)=(f(x_1),g(y_1))$. This is a Borel map, as the restriction of the Borel function $f \times g$ to the Borel subspace $X_1*Y_1 \subseteq X_1 \times Y_1$. Moreover, the above diagram is commutative. Observe that the fibers of the map $f*g$ are Cartesian products of the form $(f*g)^{-1}(x_2,y_2)=f^{-1}(x_2)\times g^{-1}(y_2)$. \\

We will assume throughout this section that $\gamma_X^{\bullet}$ and $\gamma_Y^{\bullet}$ are \textit{locally finite} systems of measures.

\begin{mydef}\label{def:fibred product BSM}
The \textbf{fibred product} of the locally finite systems of measures $\gamma_X^{\bullet}$ and $\gamma_Y^{\bullet}$, denoted $(\gamma_X * \gamma_Y)^{\bullet}$, is defined by $$\left(\gamma_X * \gamma_Y\right)^{(x_2,y_2)}=\gamma_X^{x_2} \times \gamma_Y^{y_2}.$$ More precisely, $\left(\gamma_X * \gamma_Y\right)^{(x_2,y_2)}(E)=\left(\gamma_X^{x_2} \times \gamma_Y^{y_2}\right)(E\cap (f*g)^{-1}(x_2,y_2))$, for every $(x_2,y_2)\in X_2*Y_2$ and every Borel set $E\subseteq X_1*Y_1$.
\end{mydef}

\begin{prop}\label{prop:fibred product BSM}
The fibred product $(\gamma_X * \gamma_Y)^{\bullet}$ is a locally finite system of measures on $f * g$. If $\gamma_X^{\bullet}$ and $\gamma_Y^{\bullet}$ are both locally finite BSMs, then so is $(\gamma_X * \gamma_Y)^{\bullet}$.
\end{prop}

\begin{proof}
The proof is very similar to the proof of Proposition \ref{prop:lifting BSM}. As a product of locally finite Borel measures, $(\gamma_X * \gamma_Y)^{(x_2,y_2)}$ is a well defined Borel measure for every $(x_2,y_2)\in X_2*Y_2$. By definition it is concentrated on $f^{-1}(x_2)\times g^{-1}(y_2)=(f*g)^{-1}(x_2,y_2)$.
A calculation analogous to (\ref{calc:lifting}) in Remark \ref{rem:elementary open sets} gives
\begin{eqnarray}
(\gamma_X * \gamma_Y)^{(x,y)}(E)=\gamma_X^{x}(A)\cdot\gamma_Y^{y}(B)\label{calc:fibred product}
\end{eqnarray}
for any elementary open set of the form $E=(A\times B)\cap (X_1*Y_1)$.
Therefore, using the local finiteness of $\gamma_X^{\bullet}$ and $\gamma_Y^{\bullet}$, we can find for any $(x_1,y_1) \in X_1*Y_1$ a neighborhood $(U_{x_1} \times U_{y_1}) \cap (X_1*Y_1)$ satisfying $\left(\gamma_X * \gamma_Y\right)^{(x_2,y_2)}\left((U_{x_1} \times U_{y_1}) \cap (X_1*Y_1)\right) = \gamma_X^{x_2}(U_{x_1})\cdot\gamma_Y^{y_2}(U_{y_1}) <\infty$ for all $(x_2,y_2) \in X_2*Y_2$. Thus $(\gamma_X * \gamma_Y)^{\bullet}$ is a locally finite system of finite measures.

Now assume that $\gamma_X^{\bullet}$ and $\gamma_Y^{\bullet}$ are both locally finite BSMs. We need to prove that $(\gamma_X * \gamma_Y)^{\bullet}(E)$ is a Borel function for any Borel subset $E\subseteq X_1*Y_1$, but as in Proposition \ref{prop:lifting BSM} it is sufficient to prove it for any elementary open subset $E=(A\times B)\cap (X_1*Y_1)$.
The rest of the proof uses the same arguments as Proposition \ref{prop:lifting BSM}.
\end{proof}

In order to prove that a fibred product of CSMs is a CSM, we first need a lemma. We remind that in the CSM context, spaces are assumed to be Hausdorff and locally compact.

\begin{lem}\label{lem:fibred cartesian continuity}
Let $\psi \in C_c(X_1 \times Y_1)$. The function  $\displaystyle (\xi,\eta) \mapsto \int_{X_1 \times Y_1} \psi(x,y) d\gamma_X^{\xi} d\gamma_Y^{\eta}$ is in $C_c(X_2 \times Y_2)$.
\end{lem}

\begin{proof}
Define a function $F$ on $X_2 \times Y_1$ by $(\xi,y)\mapsto \int_{X_1} \psi(x,y)d\gamma_X^{\xi}(x)$. Using (a symmetric version of) Lemma \ref{lem:continuity lemma} with $\mathcal{X} = Y_1$, $\mathcal{Y} = X_1$ and $\mathcal{Z} = X_2$, we deduce that $F \in C_c(X_2 \times Y_1)$.

Now define a function $G$ on $X_2 \times Y_2$ by $(\xi,\eta) \mapsto \int_{Y_1} F(\xi,y)d\gamma_Y^{\eta}(y)$.
Again by (a symmetric version of) Lemma \ref{lem:continuity lemma} with $\mathcal{X} = X_2$, $\mathcal{Y} = Y_1$ and $\mathcal{Z} = Y_2$, we deduce that $G \in C_c(X_2 \times Y_2)$. This is what we had to prove.
\end{proof}

\begin{prop}\label{prop:fibred continuity}
If $\gamma_X^{\bullet}$ and $\gamma_Y^{\bullet}$ are both CSMs, then so is the fibred product $(\gamma_X * \gamma_Y)^{\bullet}$.
\end{prop}

\begin{proof}
The proof is similar to that of Proposition \ref{prop:lifting CSM}. Let $\psi \in  C_c(X_1 * Y_1)$. We need to show that the function $(\xi,\eta) \mapsto \int_{X_1 * Y_1} \psi(x,y) d(\gamma_X * \gamma_Y)^{(\xi,\eta)}(x,y)$ is continuous on$X_2 * Y_2$. As argued in the proof of Proposition \ref{prop:lifting CSM}, by Tietze's Extension Theorem, there exists a function $F \in C_c(X_1 \times Y_1)$ such that $F|_{X_1 * Y_1} = \psi$.

By Lemma \ref{lem:fibred cartesian continuity}, the map $G:(\xi,\eta) \mapsto \int_{X_1 \times Y_1} F(x,y) d\gamma_X^{\xi}(x) d\gamma_Y^{\eta}(y)$ is in $C_c(X_2 \times Y_2)$. In fact, $G|_{X_2 * Y_2} \in C_c(X_2 * Y_2)$, since $X_2 * Y_2$ is closed in $X_2 \times Y_2$.

Note that the measure $\gamma_X^{\xi}$ is concentrated on $f^{-1}(\xi)$ and the measure $\gamma_Y^{\eta}$ is concentrated on $g^{-1}(\eta)$. Hence their product is concentrated on the set of $(x,y)$ satisfying $f(x) = \xi, g(y) = \eta$. For $(\xi,\eta) \in X_2 * Y_2$ we have $p_2(\xi) = q_2(\eta)$, so $p_2(f(x)) = q_2(g(y))$. Recalling that $p_1=p_2\circ f$ and $q_1=q_2\circ g$, we get $p_1(x) = p_2(f(x)) = q_2(g(y)) = q_1(y)$, i.e. $(x,y) \in X_1 * Y_1$.
We conclude that the continuous map $G|_{X_2 * Y_2}$ satisfies
\begin{eqnarray*}
G|_{X_2 * Y_2}(\xi,\eta) &=& \int_{X_1 \times Y_1} F(x,y) d\gamma_X^{\xi} (x) d\gamma_Y^{\eta} (y) \ = \ \int_{X_1 * Y_1} F(x,y) d\gamma_X^{\xi} (x)d\gamma_Y^{\eta}(y) \\ &=& \int_{X_1 * Y_1} \psi(x,y) d\gamma_X^{\xi}(x) d\gamma_Y^{\eta}(y) \ = \ \int_{X_1 * Y_1} \psi(x,y) d(\gamma_X * \gamma_Y)^{(\xi,\eta)}(x,y).
\end{eqnarray*}
This completes the proof.
\end{proof}

\begin{prop}\label{prop:fibred stability}
The properties of being positive on open sets and locally bounded are preserved under fibred products.
\end{prop}

\begin{proof}
The proof is very similar to its counterpart for lifting in Proposition \ref{prop:lifting stability}. Assume that $\gamma_X^{\bullet}$ and $\gamma_Y^{\bullet}$ are positive on open sets. To prove that $(\gamma_X*\gamma_Y)^{\bullet}$ is positive on open sets is suffices to consider elementary open sets. Fix $(x,y)\in X_2*Y_2$ and let $E=(A\times B)\cap (X_1*Y_1)$ be an elementary open set such that $E\cap(f*g)^{-1}(x,y)\neq \emptyset$. This implies that $A\cap f^{-1}(x)\neq \emptyset$ and $B\cap g^{-1}(y)\neq \emptyset$, hence $\gamma_X^x(A)>0$ and $\gamma_Y^y(B)>0$. Using calculation (\ref{calc:fibred product}) from Proposition \ref{prop:fibred product BSM} we obtain $(\gamma_X * \gamma_Y)^{(x,y)}(E)=\gamma_X^x(A)\cdot\gamma_Y^y(B)>0$.

Proving that the lifted system is locally bounded is similar to the proof that it is locally finite in Proposition \ref{prop:composition BSM}.
\end{proof}

Assume that we now have for $i=$1,2,3 the following three pull-back diagrams
$$\xy\xymatrix{X_i*Y_i\ar[r]\ar[d]&X_i\ar[d]_{p_i}\\
               Y_i\ar[r]^{q_i}&Z}\endxy$$
where the maps $p_i$ and $q_i$ are all continuous. Furthermore, assume that we have continuous connecting maps $\xymatrix{X_1\ar [rr]^{f_1}_{\gamma_1^{\bullet}}&&X_2}$, $\xymatrix{Y_1\ar [rr]^{g_1}_{\xi_1^{\bullet}}&&Y_2}$,  $\xymatrix{X_2\ar [rr]^{f_2}_{\gamma_2^{\bullet}}&&X_3}$ and $\xymatrix{Y_2\ar [rr]^{g_2}_{\xi_2^{\bullet}}&&Y_3}$, all endowed with locally finite systems of measures, satisfying that $p_1=p_2\circ f_1$, $q_1=q_2\circ g_1$, $p_2=p_3\circ f_2$ and $q_2=q_3\circ g_2$. Finally, assume that $\gamma_1^{\bullet}$ and $\xi_1^{\bullet}$ are locally bounded. This data allows us to implement the fibred product construction above, giving rise to the following diagram, which is commutative as a diagram of topological spaces and continuous maps:
$$
\xy \xymatrix@R=.2in@C=.2in{ & X_1*Y_1\ar[rr]^{f_1 *g_1}_{(\gamma_1 *\xi_1)^{\bullet}}\ar@{-}[dd]\ar[ddl]&&X_2*Y_2\ar[rr]^{f_2*g_2}_{(\gamma_2*\xi_2)^{\bullet}}\ar@{-}[dd]\ar[ddl]&&X_3*Y_3\ar[dddd]\ar[ddl]\\%
                          \hspace{0in}&&\hspace{0in}&&&\hspace{0in}\\%
              X_1\ar[rr]^(.65){f_1}_(.65){\gamma_1^{\bullet}}\ar[dddd]^(.3){p_1}   &{\ }^{\ }\ar[dd]&X_2\ar[rr]^(.65){f_2}_(.65){\gamma_2^{\bullet}}\ar[dddd]^(.3){p_2}&{\ }^{\ } \ar[dd]&X_3\ar[dddd]^(.3){p_3}&\vspace{-2in}\\%
              &&&&&\\%
                          &Y_1\ar@{-}[r]^(.65){g_1}_(.65){\xi_1^{\bullet}}  \ar[ddl]^{q_1} &\hspace{.06in}\ar[r]&Y_2\ar@{-}[r]^(.65){g_2}_(.65){\xi_2^{\bullet}}\ar[ddl]^{q_2}&\hspace{.06in}\ar[r]&Y_3\ar[ddl]^{q_3}\\%
                          &&&&&\\%
             Z\ar[rr]^{\scriptstyle\rm id}&&Z\ar[rr]^{\scriptstyle\rm id}&&Z&} \endxy%
$$
Loosely speaking, the following proposition states that fibred products and compositions of systems of measures, commute.

\begin{prop}\label{prop:fibred products commute with compositions}
In the above setting,
$$[(\gamma_2 *\xi_2)\circ (\gamma_1*\xi_1)]^{\bullet}=[(\gamma_2\circ\gamma_1)*(\xi_2\circ\xi_1)]^{\bullet}.$$
\end{prop}

\begin{proof}
Both $[(\gamma_2 *\xi_2)\circ (\gamma_1*\xi_1)]^{\bullet}$ and $[(\gamma_2\circ\gamma_1)*(\xi_2\circ\xi_1)]^{\bullet}$ are systems of measures on the map from $X_1*Y_1$ to $X_3*Y_3$, defined by $(x_1,y_1)\mapsto (f_2(f_1(x_1)), g_2(g_1(y_1)))$. By Proposition \ref{prop:fibred product BSM}, $(\gamma_1*\xi_1)^{\bullet}$ and $(\gamma_2*\xi_2)^{\bullet}$ are locally finite, the former being also locally bounded by Proposition \ref{prop:fibred stability}. Thus, by Lemma \ref{lem:condition for composition locally finite}, $[(\gamma_2 *\xi_2)\circ (\gamma_1*\xi_1)]^{\bullet}$ is a locally finite system of measures. Moreover, by Lemma \ref{lem:condition for composition locally finite}, $(\gamma_2\circ\gamma_1)^{\bullet}$ and $(\xi_2\circ\xi_1)^{\bullet}$ are locally finite, implying in turn that $[(\gamma_2\circ\gamma_1)*(\xi_2\circ\xi_1)]^{\bullet}$ is locally finite by Proposition \ref{prop:fibred product BSM}.

Fix $(x_3,y_3)\in X_3*Y_3$. For any Borel set $E \subseteq X_1*Y_1$, define
$$\mu (E) = [(\gamma_2 *\xi_2)\circ (\gamma_1*\xi_1)]^{(x_3,y_3)} (E) \quad \text{and} \quad \nu(E)=[(\gamma_2\circ\gamma_1)*(\xi_2\circ\xi_1)]^{(x_3,y_3)}(E).$$ Being extracted from locally finite systems of measures, $\mu$ and $\nu$ are locally finite measures on $X_1*Y_1$.

Next, let $E=(A\times B)\cap (X_1*Y_1)$ be an elementary open set. Using the definitions of fibred products and compositions, along with Fubini's theorem, we get
\begin{eqnarray}
\mu (E) &=& [(\gamma_2 *\xi_2)\circ (\gamma_1*\xi_1)]^{(x_3,y_3)}(E)\nonumber\\
&=&\int_{X_2*Y_2}(\gamma_1*\xi_1)^{(x_2,y_2)}(E)\ d(\gamma_2*\xi_2)^{(x_3,y_3)}(x_2,y_2)\nonumber\\
&=&\int_{Y_2}\int_{X_2}(\gamma_1*\xi_1)^{(x_2,y_2)}(E)\ d\gamma_2^{x_3}(x_2)d\xi_2^{y_3}(y_2)\nonumber\\
&=&\int_{Y_2}\int_{X_2} \gamma_1^{x_2}(A)\xi_1^{y_2}(B) d\gamma_2^{x_3}(x_2)d\xi_2^{y_3}(y_2)\nonumber\\
&=&\left(\int_{X_2}\gamma_1^{x_2}(A)d\gamma_2^{x_3}(x_2)\right)\cdot \left(\int_{Y_2} \xi_1^{y_2}(B) d\xi_2^{y_3}(y_2)\right)\nonumber\\
&=&(\gamma_2\circ \gamma_1)^{x_3}(A)\cdot(\xi_2\circ\xi_1)^{y_3}(B)\nonumber\\&=&[(\gamma_2\circ \gamma_1)*(\xi_2\circ\xi_1)]^{(x_3,y_3)}(E)\nonumber\\
&=& \nu(E).\nonumber
\end{eqnarray}

Finally, let $\{A_n\}_{n=1}^{\infty}$ and $\{B_m\}_{m=1}^{\infty}$ be bases for the topology of $X_1$ and $Y_1$ respectively. The collection ${\mathcal B}=\{(A_n\times B_m)\cap(X_1 *Y_1)\}_{n,m}$ is a countable basis for the topology of $X_1 *Y_1$ consisting of elementary open sets. Moreover, we have seen that $\mu$ and $\nu$ agree on finite intersections of sets in ${\mathcal B}$, since these are also of the form $E=(A\times B)\cap (X_1*Y_1)$. We can now apply lemma \ref{lem:mu=nu for every E Borel} to the basis $\mathcal{B}$ and the locally finite measures $\mu$ and $\nu$, and conclude that $\mu(E)=\nu(E)$ for any Borel set $E\subseteq X_1*Y_1$. Since $(x_3,y_3)$ was arbitrary, this completes the proof.
\end{proof}

\section{Disintegration}\label{sec:disintegration}

Disintegration of measures (sometimes called decomposition) has received vast attention in the literature. The purpose of presenting it here is limited to providing versions and derivatives of the fundamental result (Theorem \ref{thm:disintegration}, Corollary \ref{cor:disintegration locally finite} and Proposition \ref{prop:locally bounded disintegration}) which are consistent with our approach and terminology and suitable for our needs. This is why we chose to quote Fabec \cite{fabec-book}, rather than probably the most original source (von Neumann \cite{von-Neumann}) or alternatively more generalized versions. We do refer the reader interested in tracing the theorem historically to Ramsay (\cite{ramsay71}, page 264), which in turn cites Mackey, Halmos, and ultimately von Neumann. Throughout this section we shall assume all spaces to be second countable, locally compact and Hausdorff.

\begin{mydef}\label{def:measure-class-preserving}
Let $(X,\mu)$ and $(Y,\nu)$ be measure spaces. We will say that a Borel map $f:X \rightarrow Y$ is \textbf{measure-preserving} if $f_*\mu = \nu$. We will say that $f$ is \textbf{measure-class-preserving} if $f_*\mu \sim \nu$.
\end{mydef}

In the above definition $f_*$ is the push-forward, defined for any Borel set $F \subset Y$ by $f_* \mu(F) = \mu (f^{-1}(F))$, and $\sim$ denotes equivalence of measures in the sense of being mutually absolutely continuous.

\begin{mydef}\label{def:disintegration}
Let $(X,\mu)$ and $(Y,\nu)$ be measure spaces, and let $f:X \rightarrow Y$ be a Borel map. A system of measures $\gamma^{\bullet}$ on $f$ will be called a \textbf{disintegration} of $\mu$ with respect to $\nu$ if $\displaystyle \mu (E) = \int_Y \gamma^y (E) d\nu (y)$ for every Borel set $E \subseteq X$.
\end{mydef}

\begin{lem}\label{lem:probability disintegration implies measure preserving}
If $\gamma^{\bullet}$ is a system of probability measures on $f$ which is a disintegration of $\mu$ with respect to $\nu$, then $f$ is measure preserving.
\end{lem}

\begin{proof}
Since $\gamma^{\bullet}$ is a system of probability measures, $\gamma^y$ is concentrated on $f^{-1}(y)$ and $\gamma^y (f^{-1}(y)) = 1$ for any $y \in Y$. Therefore, $\gamma^y (f^{-1}(F)) = \chi_{_F}(y)$ for any Borel set $F \subseteq Y$. Thus, for any Borel set $F \subseteq Y$ we have
\[
f_* \mu(F) \ = \ \mu (f^{-1}(F)) \ = \ \int_Y \gamma^y (f^{-1}(F)) d\nu (y) \ = \ \int_Y \chi_{_F}(y) d\nu (y) \ = \ \nu (F),
\]
so $f$ is measure preserving.
\end{proof}

\begin{lem}\label{lem:disintegration implies measure class preserving}
Let $\gamma^{\bullet}$ be a system of measures on $f$ which is positive on open sets. If $\gamma^{\bullet}$ is a disintegration of $\mu$ with respect to $\nu$, then $f$ is measure-class-preserving.
\end{lem}

\begin{proof}
Let $F \subseteq Y$ be a Borel set. For any $y\in Y$, we have $\gamma^y (f^{-1}(F))=\chi_{_F}(y) \cdot \gamma^y(f^{-1}(y)) = \chi_{_F}(y) \cdot \gamma^y(X)$. Therefore $$ f_* \mu(F) \ = \ \mu (f^{-1}(F)) \ = \ \int_Y \gamma^y (f^{-1}(F)) d\nu (y) \ = \ \int_Y \chi_{_F}(y)\gamma^{y}(X) d\nu (y) \ = \ \int_F \gamma^{y}(X) d\nu (y).$$
This shows that $f_* \mu$ is absolutely continuous with respect to $\nu$. Moreover, since $\gamma^{\bullet}$ is positive on open sets, $\gamma^{\bullet}(X)$ is a positive function on $Y$, and thus $f_* \mu$ is equivalent to $\nu$. We conclude that $f$ is measure-class-preserving.
\end{proof}

The converse to the previous lemmas is less trivial. The following theorem is a restatement of Theroem I.27 in \cite{fabec-book}. The original theroem requires $X$ to be a standard Borel space, which is a Polish space (i.e. a second countable topological space admitting a complete metric that generates the topology), together with its Borel $\sigma$-algebra. However, recall that our spaces are assumed to be locally compact, Hausdorff and second countable, hence they are standard Borel spaces. We refer the reader to a paper by Ramsay \cite{ramsay} for a discussion of these facts.

\begin{theorem}[\cite{fabec-book}, Theorem I.27]\label{thm:disintegration}
Let $(X,\mu)$ and $(Y,\nu)$ be spaces equipped with $\sigma$-finite measures, and let $f:X \rightarrow Y$ be a measure-class-preserving Borel map. Then there exists a BSM $\gamma^{\bullet}$ on $f$ which is a disintegration of $\mu$ with respect to $\nu$. Moreover, if $\gamma_1^{\bullet},\gamma_2^{\bullet}$ are two disintegrations, then $\gamma_1^{y}=\gamma_2^{y}$ for $\nu$-almost every $y\in Y$.
\end{theorem}

\begin{cor}\label{cor:disintegration locally finite} Let $(X,\mu)$ and $(Y,\nu)$ be spaces equipped with $\sigma$-finite measures, and let $f:X \rightarrow Y$ be a measure-class-preserving Borel map. If $\mu$ is a \emph{locally finite} measure, then there exists a \emph{locally finite} BSM $\alpha^{\bullet}$ on $f$ which is a disintegration of $\mu$ with respect to $\nu$. Moreover, if $\alpha_1^{\bullet},\alpha_2^{\bullet}$ are two disintegrations, then $\alpha_1^{y}=\alpha_2^{y}$ for $\nu$-almost every $y\in Y$.
\end{cor}

\begin{proof}
By Theorem \ref{thm:disintegration}, there exists a BSM $\gamma^{\bullet}$ on $f$ which is a disintegration of $\mu$ with respect to $\nu$, and it is unique $\nu$-almost everywhere in $Y$. Let ${\mathcal B}=\{B_n\}_{n=1}^{\infty}$ be a countable basis for the topology of $X$. Since $\mu$ is locally finite, it is straightforward to verify that the sub-collection $\{B \in \mathcal{B} ~|~ \mu(B)<\infty \}$ is itself a basis for $X$. Therefore, we can assume that all $B_n \in \mathcal{B}$ satisfy $\mu(B_n)=\int_Y\gamma^y(B_n)d\nu(y)<\infty$. It follows that $\gamma^y(B_n)<\infty$ for $\nu$-almost all $y\in Y$.

Consider the Borel sets $Y_n= \{ y\in Y ~|~ \gamma^y(B_n)= \infty \}$. By our previous argument, the sets $Y_n$ all have $\nu$-measure zero, hence so does $\bigcup_{n=1}^{\infty}Y_n$. We denote $Y' = Y\setminus \bigcup_{n=1}^{\infty}Y_n$ and define a new BSM $\alpha^{\bullet}$ on $f$ by $$\alpha^y(E)= \begin{cases} \gamma^y(E)& y \in Y' \\ 0 & y \notin Y'  \end{cases}$$ for any Borel set $E \subseteq X$. It is easy to verify that $\alpha^{\bullet}$ is  indeed a BSM on $f$. Moreover, since $\alpha^{y}(E)=\gamma^y(E)$ for $\nu$-almost all $y\in Y$, it follows that $\alpha^{\bullet}$ is also a disintegration of $\mu$ with respect to $\nu$, and the uniqueness $\nu$-almost everywhere in $Y$ holds for $\alpha^{\bullet}$.

It remains to show that $\alpha^{\bullet}$ is locally finite. For any $x\in X$, let $B_n \in {\mathcal B}$ be a neighborhood of $x$. Since $$\alpha^y(B_n)= \begin{cases} \gamma^y(B_n)& y \in Y' \\ 0 & y \notin Y'  \end{cases}$$ it follows that $\alpha^y(B_n)<\infty$ for all $y\in Y$. Thus $\alpha^{\bullet}$ is locally finite, and the proof is complete.
\end{proof}

The next lemma, which is rather elementary, is required for the proof of Proposition \ref{prop:locally bounded disintegration} below. Lacking a formal reference, we include the proof, which is adapted from lecture notes found on the homepage of Gabriel Nagy.

\begin{lem}\label{lem:easy Radon-Nikodym}
Let $\mu,\nu$ be finite measures on a measurable space $(Y, \Sigma)$. Then the Radon-Nikodym derivative $h=d\mu / d\nu$ exists and belongs to $L^{\infty}(Y,\nu)$ if and only if there is a constant $C\geq 0$ such that $\mu(E)\leq C\cdot \nu(E)$ for all $E \in \Sigma$.
\end{lem}

\begin{proof}
Suppose that the Radon-Nikodym derivative $h\!=\!d\mu / d\nu$ exists and is in $L^{\infty}(Y,\nu)$. Then for all $E \in \Sigma$ we have $\mu(E)=\int_E h d\nu \leq \|h\|_{\infty}\cdot \nu(E).$

Conversely, assume that there is a constant $C$ such that $\mu(E)\leq C\cdot \nu(E)$ for all $E \in \Sigma$. A standard argument using simple functions and the Monotone Convergence Theorem, yields $\int_Y f d\mu \leq C\cdot \int_Y f d\nu$, for any measurable function $f:Y\rightarrow [0,\infty]$. It follows that the identity map $f \mapsto f$ is a continuous function $L^1(Y,\nu)\rightarrow L^1(Y,\mu).$ Moreover, we have a composition of continuous linear functions
$L^2(Y,\nu) \hookrightarrow L^1(Y,\nu) \rightarrow L^1(Y,\mu) \rightarrow \mathbb{R}$ given by $f \mapsto f \mapsto f \mapsto \int_Y f\ d\mu$, which gives rise to a continuous linear functional on $L^2(Y,\nu)$. Since $L^2(Y,\nu)$ is a Hilbert space, there exists a function $h \in L^2(Y,\nu)$ such that
$\int_Y fh d\nu = \int_Y f d\mu$ for any $f \in L^2(Y,\nu)$. Setting $f=\chi_{_E}$ we get $\mu(E) = \int_E h d\nu$ for any $E\in \Sigma$, hence $h = d\mu / d\nu$. Denote $A_n=\left\{y\in Y: h(y)\geq C+\frac{1}{n}\right\}$. Then $\mu(A_n)= \int_{A_n} h d\nu \geq \left(C+\frac{1}{n}\right)\nu(A_n) \geq \left(1+\frac{1}{nC}\right)\mu(A_n)$, from which it follows that $0 \geq \frac{1}{nC}\mu(A_n)$, so $\mu(A_n)=0$. By the above inequality, this implies that $\nu(A_n)=0$. Therefore $A=\bigcup_{n=1}^{\infty}A_n=\left\{y\in Y: h(y)> C\right\}$ also satisfies $\nu(A)=0$. Thus $\|h\|_{\infty}\leq C$ and in particular $h \in L^{\infty}(Y,\nu)$.
\end{proof}

The following proposition provides a useful criterion for the existence of a disintegration which is locally bounded. Note that it requires the map $f$ to be continuous.

\begin{prop}\label{prop:locally bounded disintegration}
Let $(X,\mu)$ and $(Y,\nu)$ be spaces equipped with locally finite measures and let $f:X\rightarrow Y$ be a measure class preserving continuous map. The map $f$ admits a disintegration $\alpha^{\bullet}$ which is locally bounded if and only if for any compact set $K\subseteq X$ there exists a constant $C_{_K}$ such that for all Borel sets $E\subseteq Y$, $$\mu(K\cap f^{-1}(E))\leq C_{_K}\cdot \nu(E).$$
\end{prop}

\begin{proof}
Recall that our spaces are always assumed to be locally compact, Hausdorff and second countable, and as such, every locally finite measure is $\sigma$-finite. By Corollary \ref{cor:disintegration locally finite}, $f$ admits a disintegration $\alpha^{\bullet}$ of $\mu$ with respect to $\nu$, which is unique $\nu$-almost everywhere in $Y$. Note that the system $\alpha^{\bullet}$ can be taken to be locally bounded, or equivalently bounded on compact sets, if and only if for any compact $K \subseteq X$, $\alpha^{\bullet}(K)$ is in $L^{\infty}(Y,\nu)$, i.e. essentially bounded.

For every compact set $K\subseteq X$, consider the measure $\mu_{_K}$ on $Y$ defined by $\mu_{_K}(E):=\mu(K\cap f^{-1}(E)),$
for all Borel sets $E\subseteq Y$. The measure $\mu_{_K}$ is finite since $\mu$ is locally finite, and moreover, since $f$ is measure class preserving, $\mu_{_K}$ is absolutely continuous with respect to $\nu$. Let $h_{_K} = d\mu_{_K}/d\nu$ denote the Radon-Nikodym derivative. Thus, for any Borel subset $E\subseteq Y$, we have $$\mu(K\cap f^{-1}(E))=\mu_{_K}(E)=\int_E h_{_K}(y)\ d\nu(y).$$
On the other hand, $$\mu(K\cap f^{-1}(E))=\int_Y \alpha^y(K\cap f^{-1}(E))\ d\nu(y)=\int_E \alpha^y(K)\ d\nu(y).$$
Therefore, $\int_E h_{_K}\ d\nu=\int_E \alpha^{\bullet}(K)\ d\nu$ for any $E$, hence $h_{_K}=\alpha^{\bullet}(K)$, $\nu$-almost everywhere in $Y$.

Let $\nu_{_K}$ be another measure on $Y$, defined by $\nu_{_K}(E)=\nu(E\cap f(K)).$ The measure $\nu_{_K}$ is finite, since $K$ is compact, $f$ is continuous and $\nu$ is locally finite. Moreover, $\mu_{_K}$ is absolutely continuous with respect to $\nu_{_K}$:
\begin{eqnarray*}
\nu_{_K}(E) \!=\!0 & \Rightarrow & \nu(E\cap f(K)) \!=\!0 \ \ \Rightarrow \ \ f_*\mu(E\cap f(K)) \!=\!0 \ \  \Rightarrow \ \  \mu(f^{-1}(E\cap f(K))) \!=\!0 \\
& \Rightarrow & \mu(f^{-1}(E) \cap f^{-1}(f(K))) \!=\!0 \ \  \Rightarrow \ \   \mu(f^{-1}(E) \cap K) \!=\!0 \ \  \Rightarrow \ \  \mu_{_K}(E) \!=\!0.
\end{eqnarray*}
The Radon-Nikodym derivative $d\mu_{_K} / d\nu_{_K}$ is equal $\nu_{_K}$-almost everywhere to $h_{_K}$, since
\begin{eqnarray*}
\mu_{_K}(E) &=& \mu(K\cap f^{-1}(E)) \ = \ \mu(K\cap f^{-1}(E)\cap f^{-1}(f(K))) \ = \ \mu(K\cap f^{-1}(E\cap f(K))) \\
&=& \int_{E\cap f(K)} h_{_K}(y)\ d\nu(y) \ = \ \int_E h_{_K}(y)\ d\nu_{_K}(y).
\end{eqnarray*}
In particular, $h_{_K} \!=\!0$ $\nu$-almost everywhere outside $f(K)$, since $\int_{E\cap f(K)} h_{_K}(y)\ d\nu(y) = \mu_{_K}(E) = \int_E h_{_K}(y)\ d\nu(y)$ for any Borel subset $E\subseteq Y$. It follows that $h_{_K} \in L^{\infty}(Y,\nu) \Leftrightarrow h_{_K} \in L^{\infty}(Y,\nu_{_K}).$

We now apply Lemma \ref{lem:easy Radon-Nikodym} to the finite measures $\mu_{_K}$ and $\nu_{_K}$: The Radon-Nikodym derivative $d\mu_{_K} / d\nu_{_K}  \in L^{\infty}(Y,\nu_{_K})$ if and only if there is a constant $C\!=\!C_{_K}$ such that $\mu_{_K}(E) \leq C_{_K} \cdot \nu_{_K}(E)$ for all $E\subseteq Y$ Borel. Equivalently: $h_{_K} \in L^{\infty}(Y,\nu)$ if and only if there is a constant $C_{_K}$ such that $\mu(K\cap f^{-1}(E))\leq C_K\cdot \nu(E\cap f(K))$.
Observe that the condition $\mu(K\cap f^{-1}(E))\leq C_K\cdot \nu(E\cap f(K)), \forall E$ is equivalent to the condition $\mu(K\cap f^{-1}(E))\leq C_K\cdot \nu(E), \forall E$. Indeed, the latter implies the former by taking $E \cap f(K)$. Recalling that $h_{_K}=\alpha^{\bullet}(K)$ $\nu$-almost everywhere in $Y$, we conclude that $\alpha^{\bullet}(K) \in L^{\infty}(Y,\nu)$ if and only if there is a constant $C_{_K}$ such that $\mu(K\cap f^{-1}(E))\leq C_K\cdot \nu(E), \forall E$. This completes the proof.
\end{proof}

\section{Systems of measures for groupoids}\label{sec:groupoids}

Terminology in the groupoid literature is often a source for confusion. In this section we give a definition of Haar systems using the terminology we have adopted above, and show that it coincides with the standard definitions.

\begin{mydef}\label{def:system of measures on G}
Let $G$ be a topological groupoid. A system of measures $\lambda^{\bullet}$ on the range map $r:G \rightarrow G^{(0)}$ is said to be a \textbf{system of measures on $G$}.
\end{mydef}

\begin{mydef}\label{def:left_invariant}
A system of measures $\lambda^{\bullet}$ on $G$ is called \textbf{left invariant} if for every $x \in G$ and for every Borel subset $E \subseteq G$, $$\lambda^{d(x)} (E) = \lambda^{r(x)}\left(x \cdot (E \cap G^{d(x)})\right).$$
\end{mydef}

\begin{lem}\label{lem:equivalent left invariance}
A system of measures $\lambda^{\bullet}$ on $G$ is left invariant if and only if for any $x \in G$ and every non-negative Borel function $f$ on $G$, $$\int f(xy)d\lambda^{d(x)}(y) = \int f(y)d\lambda^{r(x)}(y).$$
\end{lem}

\begin{proof}
Assume $\lambda^{\bullet}$ is left invariant. Fix $x \in G$, and note that $ y \in x \cdot (E \cap G^{d(x)})
\Leftrightarrow x^{-1}y \in E \cap G^{d(x)}$. Therefore, for any Borel set $E$,
\begin{eqnarray*}
\int_G \chi_{_E}(y) d\lambda^{d(x)}(y) &=& \lambda^{d(x)}(E) \ = \ \lambda^{r(x)}\left(x \cdot (E \cap G^{d(x)})\right) \ = \ \int_G \chi_{_{x \cdot (E \cap G^{d(x)})}} (y) d\lambda^{r(x)}(y) \\ &=& \int_G \chi_{_{E \cap G^{d(x)}}} (x^{-1}y) d\lambda^{r(x)}(y) \ = \ \int_G \chi_{_E} (x^{-1}y) d\lambda^{r(x)}(y).
\end{eqnarray*}
Replacing $x$ with $x^{-1}$ we get $\int_G \chi_{_E}(y) d\lambda^{r(x)}(y) \ = \ \int_G \chi_{_E} (xy) d\lambda^{d(x)}(y)$. Passing, as usual, from characteristic functions to any non-negative Borel function, we obtain that for any $x \in G$ and for every non-negative Borel function $f$, $\int_G f(y) d\lambda^{r(x)}(y) \ = \ \int_G f (xy) d\lambda^{d(x)}(y)$ as claimed.

The converse is obtained by reversing the arguments.
\end{proof}

\begin{lem}\label{lem:CSM is left invariant for any f in Cc}
A CSM $\lambda^{\bullet}$ on $G$ is left invariant if and only if for any $x \in G$ and every function $f\in C_c(G)$, $$\int f(xy)d\lambda^{d(x)}(y) = \int f(y)d\lambda^{r(x)}(y).$$
\end{lem}

\begin{proof}
Assume first that $\lambda^{\bullet}$ is a left invariant CSM on $G$. By Proposition \ref{prop:CSM is BSM}, $\lambda^{\bullet}$ is a BSM, and by Lemma \ref{lem:equivalent left invariance} we have that for any $x \in G$ and every non-negative Borel function $f$ on $G$, $\int f(xy)d\lambda^{d(x)}(y) = \int f(y)d\lambda^{r(x)}(y).$ In particular this holds for any non-negative $f\in C_c(G)$. The usual decomposition of a general complex-valued $f\in C_c(G)$ as $f=f_1+if_2$, and further as $f_k = (f_k)_+ -(f_k)_-$, yields the property for any $f\in C_c(G)$.

Conversely, if a CSM satisfies $\int f(xy)d\lambda^{d(x)}(y) = \int f(y)d\lambda^{r(x)}(y)$ for any $f\in C_c(G)$, then in particular the property holds for non-negative $f\in C_c(G)$. By approximating characteristic functions of open sets by continuous ``bump" functions and using a standard Monotone Convergence Theorem argument, we obtain that $\int \chi_{_A}(xy)d\lambda^{d(x)}(y) = \int \chi_{_A}(y)d\lambda^{r(x)}(y)$ for any open subset $A \subseteq X$ and any $x \in G$. By means of a calculation similar to that of Lemma \ref{lem:equivalent left invariance}, we deduce that $\lambda^{d(x)} (A) = \lambda^{r(x)}\left(x \cdot (A \cap G^{d(x)})\right)$ for any open subset $A \subseteq X$ and any $x \in G$. Finally, we denote $\mu_x(A) = \lambda^{d(x)} (A)$ and $\nu_x(A) = \lambda^{r(x)}\left(x \cdot (A \cap G^{d(x)})\right)$, and apply Corollary \ref{cor:mu=nu for every A open} to the locally finite measures $\mu_x$ and $\nu_x$. We conclude that $\mu_x(E) = \nu_x(E)$ for any Borel subset $E \subseteq X$. This holds for any $x \in G$, which implies that $\lambda^{\bullet}$ is left invariant.
\end{proof}

\begin{mydef}\label{def:Haar system}
A continuous left \textbf{Haar system} for $G$ is a system of measures $\lambda^{\bullet}$ on $G$ which is continuous, left invariant and positive on open sets.
\end{mydef}

We should point out that in the groupoid literature, the definition of a continuous left Haar system for $G$ appears different than ours at first glance. Modulo minor discrepancies between various sources (see for example standard references such as \cite{muhly-book-unpublished}, \cite{paterson-book}, \cite{renault-book} and \cite{renault-anantharaman-delaroche}), it is usually defined to be a family $\lambda= \{\lambda^u : u \in G^{(0)} \}$ of positive (Radon) measures on $G$ satisfying the following properties:
\begin{enumerate}
\item
$supp(\lambda^u) = G^u$ for every $u \in G^{(0)}$;
\item
(continuity) for any $f \in C_c(G)$, the function $u \mapsto \int f d\lambda^u$ on $G^{(0)}$ is in $C_c(G^{(0)})$;
\item
(left-invariance) for any $x \in G$ and $f \in C_c(G)$, $$\int f(xy)d\lambda^{d(x)}(y) = \int f(y)d\lambda^{r(x)}(y).$$
\end{enumerate}
However, by Lemma \ref{lem:positive iff full support}, Corollary \ref{cor:CSM has compact support} and Lemma \ref{lem:CSM is left invariant for any f in Cc}, the above definition is equivalent to our Definition \ref{def:Haar system}.

\section*{Acknowledgments}

We wish to thank John Baez, Jean Renault, Baruch Solel, Christopher Walker and most of all Paul Muhly, for enlightening conversations and useful suggestions.

\end{document}